\documentclass[
final,
    11pt,
]{scrartcl}

\usepackage{amsmath}
\usepackage{amssymb}
\usepackage{amsthm}
\usepackage[final]{microtype}
\usepackage{bm}

\usepackage{mathtools}
\mathtoolsset{showonlyrefs}

\usepackage[backend=biber,style=alphabetic,doi=false,isbn=false,url=true,eprint=false,
    giveninits=true,]{biblatex}
\bibliography{literature.bib}

\usepackage[hidelinks, draft=false]{hyperref}
\usepackage{orcidlink}

\usepackage[draft,
    ngerman,
    textsize=scriptsize,
]{todonotes}

\usepackage{marginnote}
\usepackage{etoolbox}

\newcommand{\ruggedtodo}[2][]{\tikzexternaldisable\todo[#1]{#2}\tikzexternalenable }

\newtoggle{lmargin}
\newcommand{\alternatingtodo}[2][]{\iftoggle{lmargin}{\ruggedtodo[#1]{#2}\togglefalse{lmargin}}{{\let\marginpar\marginnote \reversemarginpar \ruggedtodo[#1]{#2}}\toggletrue{lmargin}}\ignorespaces }

\usepackage{enumitem}
\setlist[enumerate, 1]{label=\alph*)} \setlist[description]{font=\normalfont, wide, , labelindent=0pt}

\PassOptionsToPackage{final}{graphicx}

\usepackage{tikz}
\usepackage{pgfplots}
\pgfplotsset{compat=1.18}
\usepgfplotslibrary{external}
\pgfkeys{/pgf/images/include external/.code=\includegraphics{#1}}

\newcommand{\externalizeifable}{\makeatletter
    \ifnum\pdf@shellescape=1\relax \tikzexternalize[prefix=externalized/]\else \relax \fi \makeatother
}

\usepackage{csquotes}

\usepackage{booktabs}
\usepackage{longtable}

\usepackage{caption}
\usepackage{subcaption}
\newtheorem{theorem}{Theorem}
\newtheorem{lemma}[theorem]{Lemma}
\theoremstyle{remark}

\theoremstyle{definition}
\newtheorem{definition}[theorem]{Definition}
\newtheorem{example}[theorem]{Example}
\newtheorem{remark}[theorem]{Remark}

\RequirePackage{amsmath}

\newcommand{\smoothness}{k}
\newcommand{\xdim}{n}
\newcommand{\nswitches}{s}
\newcommand{\totaldim}{d}
\newcommand{\nineqconstrs}{p}
\newcommand{\neqconstrs}{q}
\newcommand{\totalconstrs}{m}

\newcommand{\optvar}{x}
\newcommand{\absvar}{y}
\newcommand{\switchvar}{z}
\newcommand{\totalvar}{w}

\newcommand{\absobjfct}{\varphi}
\newcommand{\objfct}{f}
\newcommand{\switchfct}{c}
\newcommand{\ineqconstr}{g}
\newcommand{\eqconstr}{h}

\newcommand{\feasible}{\mathcal{F}}

\newcommand{\switchfctdx}{Z}
\newcommand{\switchfctdy}{L}
\newcommand{\switchfctdz}{M}
\newcommand{\ineqconstrdx}{???}
\newcommand{\ineqconstrdy}{???}
\newcommand{\ineqconstrdz}{???}
\newcommand{\eqconstrdx}{???}
\newcommand{\eqconstrdy}{???}
\newcommand{\eqconstrdz}{???}

\NewDocumentCommand{\Cswitch}{ O{\smoothness} }{C_{\operatorname{sw}}^{#1}}
\NewDocumentCommand{\Cabs}{ O{\smoothness, \nswitches} }{C_{\operatorname{abs}}^{#1}}
\NewDocumentCommand{\Ceval}{ O{\smoothness, \nswitches} }{C_{\operatorname{eval}}^{#1}}

\newcommand{\actswitches}{\alpha}
\newcommand{\actineq}{\beta}
\NewDocumentCommand{\projactsw}{ O{\actswitches} }{P_{#1}}
\NewDocumentCommand{\projinactsw}{ O{\actswitches} }{Q_{#1}}
\NewDocumentCommand{\projactineq}{ O{\actineq} }{P_{#1}}
\NewDocumentCommand{\projinactineq}{ O{\actineq} }{Q_{#1}}

\newcommand{\signature}{\sigma}

\NewDocumentCommand{\signdom}{ O{\signature} }{S_{#1}}
 \newcommand{\jac}{\mathrm{D}}
\renewcommand{\switchfctdx}{\jac_1 \switchfct}
\renewcommand{\switchfctdy}{\jac_2 \switchfct}
\renewcommand{\switchfctdz}{\jac_3 \switchfct}
\renewcommand{\ineqconstrdx}{\jac_1 \ineqconstr}
\renewcommand{\ineqconstrdy}{\jac_2 \ineqconstr}
\renewcommand{\ineqconstrdz}{\jac_3 \ineqconstr}
\renewcommand{\eqconstrdx}{\jac_1 \eqconstr}
\renewcommand{\eqconstrdy}{\jac_2 \eqconstr}
\renewcommand{\eqconstrdz}{\jac_3 \eqconstr}

\newcommand{\class}{C}
\NewDocumentCommand{\contdiff}{ O{\smoothness} }{\class^{#1}}
\newcommand{\objfcts}{\contdiff(\reals^\totaldim)}
\newcommand{\switchfcts}{\contdiff_\textup{sw}(\reals^\totaldim; \reals^\nswitches)}
\newcommand{\ineqconstrs}{\contdiff(\reals^\totaldim; \reals^\nineqconstrs)}
\newcommand{\eqconstrs}{\contdiff(\reals^\totaldim; \reals^\neqconstrs)}
\newcommand{\goodconstrs}{\mathcal{G}}

\DeclarePairedDelimiterXPP{\inner}[2]{}{\langle}{\rangle}{}{\ifblank{#1}{\cdot}{#1}, \ifblank{#2}{\bullet}{#2}}
\DeclarePairedDelimiterXPP{\norm}[1]{}{\lVert}{\rVert}{}{\ifblank{#1}{\cdot}{#1}}
\DeclarePairedDelimiterXPP{\seq}[2]{}{\lparen}{\rparen}{_{#2}}{#1}
\DeclarePairedDelimiterXPP{\abs}[1]{}{\lvert}{\rvert}{}{\ifblank{#1}{\cdot}{#1}}
\DeclarePairedDelimiter{\family}{\lbrace}{\rbrace}
\DeclarePairedDelimiter{\paren}{\lparen}{\rparen}
\DeclarePairedDelimiter{\set}{\lbrace}{\rbrace}
\DeclarePairedDelimiter{\tuple}{\lparen}{\rparen}

\newcommand{\D}{\textup{D}}
\newcommand{\Id}[1]{I_{#1}}
\newcommand{\Union}{\bigcup}
\newcommand{\after}{\circ}
\newcommand{\card}[1]{\abs{#1}} \newcommand{\field}[1]{\mathbb{#1}}
\newcommand{\from}{\colon}

\newcommand{\isom}{\cong}
\newcommand{\nats}{\field{N}}
\newcommand{\ints}{\field{Z}}
\newcommand{\oneto}[1]{[#1]}
\newcommand{\pinv}{\dagger}
\newcommand{\reals}{\field{R}}

\newcommand{\transversal}{\pitchfork}
\newcommand{\tspace}[2]{T_{#1} #2}
\newcommand{\union}{\cup}
\newcommand{\with}{\;\colon\;}
\newcommand{\zeroto}[1]{\set{0, \ldots, #1}}

\DeclareMathOperator{\diag}{diag}
\DeclareMathOperator{\img}{img}
\DeclareMathOperator{\rank}{rank}
\DeclareMathOperator{\sign}{sign}

\makeatletter
\g@addto@macro\@uclclist{\psi\Psi\omega\Omega\phi\Phi}
\makeatother

\newcommand{\stratification}{\mathcal{\stratifiedset}}
\newcommand{\stratifiedset}{A}
\NewDocumentCommand{\stratum}{ o }{\IfNoValueTF{#1}{X}{\stratifiedset_{#1}}}
\newcommand{\jet}[1]{j^{#1}}
\newcommand{\strjet}[1]{j^{#1}}
\newcommand{\jval}[1]{{\bm{#1}}}
\newcommand{\enlarged}[1]{\MakeUppercase{#1}}
\newcommand{\nfcts}{{\tilde{\nswitches}}}
\newcommand{\multidx}{a}
\newcommand{\stratspace}[1]{\tspace{#1} \stratifiedset}

\newcommand{\totalfct}{\phi}
\newcommand{\largestratification}{\tilde{\stratification}}
\newcommand{\largestratifiedset}{\tilde{\stratifiedset}}
\NewDocumentCommand{\largestratum}{ o }{\IfNoValueTF{#1}{X}{\tilde{\stratifiedset}_{#1}}}
\newcommand{\largestratspace}[1]{\tspace{#1} \largestratifiedset}

\title{How Stringent is the Linear Independence Kink Qualification in Abs-Smooth Optimization?}

\author{Lukas Baumgärtner\,\orcidlink{0000-0003-1007-4815} \and Franz Bethke\,\orcidlink{0000-0002-0553-4242} \and Ganna Shyshkanova\,\orcidlink{0000-0002-0336-2803} \and Andrea Walther\,\orcidlink{0000-0002-3516-4641} }

\begin{document}

\maketitle 

\begin{abstract}
    Abs-smooth functions are given by compositions of smooth functions and the
    evaluation of the absolute value.
The linear independence kink qualification (LIKQ) is a fundamental
    assumption in optimization problems governed by these abs-smooth functions,
    generalizing the well-known LICQ from smooth optimization.
    In particular, provided that LIKQ holds it is possible to derive optimality
    conditions for abs-smooth optimization problems that can be checked in
    polynomial time.
    Utilizing tools from differential topology, namely a version of the
    jet-transversality theorem, it is shown that assuming LIKQ for all feasible
    points of an abs-smooth optimization problem is a generic assumption.
    \noindent 

\end{abstract}

\section{Introduction}

Many real-world applications
lead to tasks with nonsmooth structures challenging significantly the
corresponding analysis.
Up to now there are hardly any off-the-shelf solution algorithms or software
packages to solve such problems, which is mainly due to the lack of
computationally tractable optimality and stationarity conditions.
For this reason, researchers concentrate on certain classes of nonsmooth
problems like, e.g., semismooth functions to derive new analytical results or
novel solution approaches.

One class of nonsmooth functions that gained some attention in the past years
are functions that are defined by a suitable composition of smooth functions
and the evaluation of the absolute value function, see,
e.g.,~\cite{GW16,HKS21,HS20,WG19}.
The main motivation to consider this set of functions was an extended version
of algorithmic differentiation~\cite{Gri13} to provide generalized derivatives
in a convenient way for nonsmooth functions given as computer programs where
the arguments of absolute value evaluations are evaluated one after the other.
Conceptually, a function \(\absobjfct \from \reals^\xdim \to \reals\) is called
\emph{abs-smooth} if for every input \(x \in \reals^\xdim\) there is a unique
vector \(z \in \reals^\nswitches\) of \(\nswitches \in \nats\) intermediate
values such that the function value \(\absobjfct(x)\) can be expressed by means
of two smooth functions \(\objfct \from \reals^\xdim \times \reals^\nswitches
\times \reals^\nswitches \to \reals\) and \(\switchfct \from \reals^\xdim
\times \reals^\nswitches \times \reals^\nswitches \to \reals^\nswitches\) by
\begin{equation}
    \begin{aligned}
        \absobjfct(x) &= \objfct(x, \abs{z}, z), \label{eq:informal-abs-smooth}\\
        z &= \switchfct(x, \abs{z}, z).
    \end{aligned}
\end{equation}
Here and throughout, the application of the absolute value function \(\abs{}\)
to a vector \(z \in \reals^\nswitches\) is to be understood component wise,
i.e., \(\abs{z}\in\reals^\nswitches\) and \(\abs{z}_i = \abs{z_i}\), \(i \in
\set{1, \ldots, \nswitches}\).
The function \(\switchfct\) will be referred to as \emph{switchting function}
of \(\absobjfct\) and the vector \(z\) corresponding to the input \(x\) will be
referred to as \emph{switchting vector} for \(x\).
Each component of \(\switchfct\) captures a step a program that implements
\(\absobjfct\) takes to compute an intermediate value whose absolute value is
used in the overall computation.
As example consider \(\absobjfct(x) \coloneqq \abs{\abs{x_1} - \abs{x_2}}\)
which can be represented by \(\switchfct(x, y, z) \coloneqq (x_1, x_2, z_1 -
z_2)\) and \(\objfct(x, y, z) \coloneqq y_3\).

An abs-smooth function could serve as target function in a nonsmooth
optimization task as considered, e.g., in~\cite{GW16}.
A collection of abs-smooth functions could represent a system of nonsmooth
equations considered, e.g., in~\cite{Gri15}, or may be used to describe
nonsmooth constraints as in~\cite{HS20}.
In the literature, such functions are also called abs-normal.
However, since the representation in~\eqref{eq:informal-abs-smooth} is by no
means unique and \(\absobjfct\) is a composition of smooth functions and the
absolute value, the term abs-smooth is used throughout this paper.
Furthermore, the list of arguments of an abs-smooth function \(\absobjfct\)
varies in the various contributions.
It was shown in~\cite{Shy+25} that these different formulations are equivalent
regarding the regularity condition introduced below.

As can be seen from the equations in~\eqref{eq:informal-abs-smooth}, the
nonsmoothness is located just in the evaluation of the absolute value.
This is an important property of abs-smooth functions that can be exploited to
derive and state theoretical properties of an abs-smooth function.
A nonsmoothness occurs only if there exists an index \(i\), \(1 \le i \le
\nswitches\), with \(z_i=0\), and probably \(z_i\) switches the sign in an
appropriate neighborhood of this point.
Motivated by the graphical representation in low dimensions, therefore the sets
of points \(x \in \reals^\xdim\), where at least one component of the
corresponding vector \(z\) is equal to zero, are called kinks.
In~\cite{GW16}, the \emph{linear independence kink qualification (LIKQ)} that
is detailed below was introduced.
Showing a close relation to the linear independence constraint qualification
(LICQ) in smooth optimization, LIKQ allows to formulate optimality conditions
that can be verified with polynomial complexity~\cite{GW16, HS20}.
In~\cite{GW16, WG19}, several examples were presented to illustrate the
prerequisites required such that LIKQ holds.
In~\cite[p.~8]{GW16} the authors write: \enquote{In general, there is no reason
why LIKQ should be violated and locally it can always be achieved by an
arbitrary small perturbation \ldots}, however, the argument given there is
rather informal.

To formalize the concept of a property that \enquote{usually} holds for a
problem class, said class is typically parametrized along the defining
functions of a problem.
A problem instance is then a point in the vector space of possible problems and
the property is said to hold, if it holds at all feasible points.

For smooth optimization, the question whether LICQ is \enquote{usually} true
was analyzed in different contributions.
In~\cite{SR79}, the authors consider a condition like LICQ as \emph{generic},
if it holds for almost all problem instances.
This concept has the drawback that the set, where the condition holds, might be
closed such that a small perturbation may lead to a situation where the
condition is no longer true.
For this reason, and more or less in parallel, in~\cite{JT79} the authors
assumed high regularity of the involved functions and used the Whitney topology
to prove that the set of problems where LICQ holds everywhere is dense and
open.
This is a very strong notion of genericity as it ensures that an arbitrarily
small perturbation to the defining functions of a problem instance leads to a
situation in which the LICQ holds at all feasible points, and that this
situation in turn is stable with respect to further perturbations.

Since the concept of abs-smooth functions was motivated originally by
algorithmic differentiation where the smooth components are usually
\(C^\infty\) functions, this paper follows the lines of Jongen and co-workers
to prove that LIKQ is generic in the sense that the set of problem
formulations where this property holds is dense and open under the assumption
that the involved functions are very regular.
Given a \enquote{random} abs-smooth problem, it is then not a strong
restriction to assume that LIKQ holds in particular at a local minimizer.
For mathematical programs with complementarity constraints (MPCCs), it was shown
in~\cite{SS01} that MPCC-LICQ is generic in the sense of Jongen, i.e., the
proof is based on Sard's theorem.
In~\cite{HKS21}, the authors prove that abs-smooth nonlinear optimizations
problems are equivalent to a certain class of MPCCs.
However, this class is just a subset of the MPCCs considered in~\cite{SS01},
and hence, one can not apply the general perturbations that are needed in the
proofs used in~\cite{SS01} to show that MPCC-LICQ is generic for the class of
MPCCs that are equivalent to abs-smooth nonlinear optimizations problems.
Therefore, these two results can not be easily combined to show the genericity
of LIKQ for the abs-smooth problem class.

Genericity of suitable constraint qualifications for other nonsmooth
optimization problem classes have been shown and applied in the critical
point theory in~\cite{JRS09, DSS12, DJS13, LS22}.
Specifically,~\cite{DJS13} invokes the structured jet-transversality theorem
of~\cite{Gün08} for Nash games; a tool that will be crucial for the arguments
presented below.

The rest of the paper is structured as follows.
Section~\ref{sec:abs-smooth} introduces the basic concepts and notations for
abs-smooth optimization problems while Section~\ref{sec:topo} presents some
definitions and results from differential topology that are required for the
application of the structured jet-transversality theorem of~\cite{Gün08} in the
context of abs-smooth problems.
The established formalism is used in Section~\ref{sec:likq-as-transversality}
to encode LIKQ as a transversality condition.
Finally, in Section~\ref{sec:generLIKQ}, the genericity of LIKQ is shown.
Section~\ref{sec:conclusion} gives an outlook to further
research questions.

To derive the theoretical results of this paper, the following
notation is used.
For a finite set \(M\) the cardinality of \(M\) is denoted by \(\card{M}\).
The set of positive integers up to \(n\in\nats\) is denoted by \(\oneto{n}
\coloneqq \set{1, \ldots, n}\).
In particular \(\oneto{0} = \emptyset\).
For a point \(x\) on a manifold \(A\) the tangent space of \(A\) at \(x\) is
denoted by \(\tspace{x}{A}\).

Let \(i, n \in \nats\), \(i \le n\) and let \(X_1, \ldots, X_n, Y\) be finite
dimensional, real vector spaces,
then the \(i\)-th partial derivative of a differentiable function \(f \colon
X_1 \times \cdots \times X_n \to Y\) is denoted by \(\partial_i f\), the
corresponding Jacobian with respect to the canonical basis by \(\D_i f\), and
the full Jacobian by \(\D f\).
Is \(f\) is smooth enough, \(\ell \in \nats\) and \(\multidx \in
\oneto{n}^\ell\) a multi-index, then \(\partial_\multidx f \coloneqq
\partial_{\multidx_\ell} \cdots \partial_{\multidx_1} f\).

For \(n\in\nats\) the identity matrix in \(\reals^{n \times n}\) is denoted by
\(\Id{n}\).
 
\section{Prerequisites from abs-smooth optimization}\label{sec:abs-smooth}
The properties of an abs-smooth function \(\absobjfct\) to a large extent
depend on the function \(\switchfct\).
The rather involved definition of \(\switchfct\) presented below already
prepares the ground for the main result of this paper and takes the fact into
account that the intermediate value \(z_i\) can only be influenced by the
intermediate value \(z_j\) if \(j < i\) which reflects the nature of a computer
program.

Let from here on onward \(\smoothness \in \nats \union \set{\infty}\) and
\(\xdim, \nswitches \in \nats\) be given and let \(\totaldim \coloneqq \xdim +
\nswitches + \nswitches\) and \(\totaldim_i \coloneqq \xdim + 2(i -1)\) for \(i
\in \oneto{\nswitches}\).

\begin{definition}[Switching functions]
    Given an \(\nswitches\)-tuple of functions \(\tuple{\switchfct_1,
    \switchfct_2, \ldots, \switchfct_s}\) with \(\switchfct_i \in
    \contdiff(\reals^{\totaldim_i})\) for \(i \in \oneto{\nswitches}\), the
    function \(\switchfct \from \reals^\totaldim \isom \reals^\xdim \times
    \reals^\nswitches \times \reals^\nswitches \to \reals^\nswitches\) defined
    by
    \begin{equation}
        \label{eq:def-switch-fct}
        \switchfct(x, y, z) \coloneqq
        {(\switchfct_1(x),
            \switchfct_2(x, y_1, z_1),
\ldots,
            \switchfct_\nswitches(x, \tuple{y_1, \ldots, y_{\nswitches-1}}, \tuple{z_1, \ldots, z_{\nswitches-1}})
        )}.
    \end{equation}
    is called a \emph{switching function} of class \(\contdiff\).
The set of all switching functions is denoted by
    \(\switchfcts \subsetneq \contdiff(\reals^\totaldim; \reals^\nswitches)\).
\end{definition}
It is important to note that, for a given \(\switchfct \in \switchfcts\) and
\(x \in \reals^\xdim\), by construction there is a unique solution \(z \in
\reals^\nswitches\) to the so-called \emph{switching equation}
\begin{equation}
    \label{eq:def-switch-eq}
    z = \switchfct(x, \abs{z}, z)
\end{equation}
by simply consecutively evaluating
\begin{equation}
    z_i = \switchfct_i(x, \tuple{\abs{z_1}, \ldots, \abs{z_{i-1}}}, \tuple{z_1, \ldots, z_{i-1}}),
\end{equation}
for all \(i \in \oneto{\nswitches}\), where for \(i = 1\) this is to be
understood as \(z_1 = c_1(x)\).
In particular, the Jacobians
\(\switchfctdy(x, \abs{z}, z) \in \reals^{\nswitches \times \nswitches}\)
and
\(\switchfctdz(x, \abs{z}, z) \in \reals^{\nswitches \times \nswitches}\)
are lower triangular matrices.

Let \(\neqconstrs, \nineqconstrs \in \nats\). For a given switching function
\(\switchfct \in \switchfcts\) and functions \(\ineqconstr \in \ineqconstrs\)
and \(\eqconstr \in \eqconstrs\) let \(\feasible(\switchfct, \ineqconstr,
\eqconstr)\) be defined by
\begin{equation}\label{eq:def-feasible}
    \feasible(\switchfct, \ineqconstr, \eqconstr)
    \coloneqq \set{\tuple{x, z} \in \reals^\xdim \times \reals^\nswitches
                   \colon z = \switchfct(x, \abs{z}, z),
                   \ineqconstr(x, \abs{z}, z) \ge 0,
                   \eqconstr(x, \abs{z}, z) = 0
               }.
\end{equation}
Together with \(\objfct \in \objfcts\) an abs-smooth optimization problem reads
\begin{equation}\label{eq:problem}
    \operatorname{minimize}\quad
    \objfct(x, \abs{z}, z)
    \quad\text{s.t.}\quad
    \tuple{x, z} \in \feasible(\switchfct , \ineqconstr, \eqconstr).
\end{equation}
The pair of functions \((\objfct, \switchfct) \in \objfcts \times \switchfcts\)
can be used to represent a function mapping \(\reals^\xdim\) to \(\reals\).
Given \(x \in \reals^\xdim\) the value \((\objfct, \switchfct)[x] \in \reals\)
is obtained by first computing the unique \(z \in \reals^s\) according to the
switching equation~\eqref{eq:def-switch-eq} and then evaluating
\begin{equation}
    (\objfct, \switchfct)[x] \coloneqq \objfct(x, \abs{z}, z).
\end{equation}
The above notation should indicate the difference to the usual component wise way of
interpreting the evaluation for a tuple of functions, i.e, \((\objfct,
\switchfct)(\totalvar) = (\objfct(\totalvar), \switchfct(\totalvar)) \in
\reals^{\nswitches + 1}\) for \(\totalvar \in \reals^\totaldim\).
If for a given function \(\absobjfct \from \reals^\xdim \to \reals\) it holds that
\(\absobjfct(x) = (\objfct, \switchfct)[x]\) for all \(x \in \reals^\xdim\), then
\(\absobjfct\) is called \emph{abs-smooth} and \((\objfct, \switchfct)\) is called an
\emph{evaluation procedure} of \(\absobjfct\).
The set of abs-smooth functions, i.e., the set of functions \(\absobjfct \from
\reals^\xdim \to \reals\) for which there is an evaluation procedure
\((\objfct, \switchfct) \in \objfcts \times \switchfcts\), is denoted by
\(\Cabs(\reals^\xdim)\).
By extending the switching function \(\switchfct\), it is possible to also
represent functions \(\absobjfct_\ineqconstr \from \reals^\xdim \to
\reals^\nineqconstrs\) and \(\absobjfct_\eqconstr \from \reals^\xdim \to
\reals^\neqconstrs\) whose evaluation involves the absolute value function such
that \(\absobjfct_\ineqconstr(x) = (\ineqconstr, \switchfct)[x]\) and
\(\absobjfct_\eqconstr(x) = (\eqconstr, \switchfct)[x]\).
For all possible evaluation procedures that represent \(\absobjfct\),
\(\absobjfct_\ineqconstr\) and \(\absobjfct_\eqconstr\), the optimization problem
in~\eqref{eq:problem} is equivalent to the problem
\begin{equation}\label{eq:abs-smooth-problem}
    \operatorname{minimize}\quad\absobjfct(x)
    \quad\text{s.t.}\quad
    x \in \set{\tilde{x} \in \reals^\xdim \with \absobjfct_\ineqconstr(\tilde{x}) \ge 0, \absobjfct_\eqconstr(\tilde{x}) = 0}.
\end{equation}

The LIKQ condition is useful to characterize local optima to~\eqref{eq:problem}
or~\eqref{eq:abs-smooth-problem} in terms of certain KKT type conditions.
In order to define LIKQ it is useful to introduce further notation for the
Jacobian matrices of the switching equation in~\eqref{eq:def-switch-eq}, the
inequality constraints \(\ineqconstr\) and the equality constraints \(\eqconstr\)
with respect to the independent variable \(x\).
To that end, let for a given \(z\in\reals^\nswitches\) the matrix of signs of
\(z\) be denoted by \(\Sigma(z) \coloneqq \diag(\sign(z)) \in \set{-1, 0,
1}^{\nswitches \times \nswitches}\) and, given \(x\in\reals^\xdim\), define
\begin{equation}
    \label{eq:def-jacobians}
    \begin{aligned}
        J_{z}(x, z)
            &\coloneqq {(\Id{\nswitches} - \switchfctdy(x, \abs{z},z)\Sigma(z) - \switchfctdz(x, \abs{z},z))}^{-1} \switchfctdx(x, \abs{z},z),\\
        J_{\ineqconstr}(x, z)
            &\coloneqq \ineqconstrdx(x, \abs{z}, z)
            + (\ineqconstrdy(x, \abs{z}, z) \Sigma(z) + \ineqconstrdz(x, \abs{z}, z)) J_{z}(x, z),\\
        J_{\eqconstr}(x, z)
            &\coloneqq \eqconstrdx(x, \abs{z}, z)
            + (\eqconstrdy(x, \abs{z}, z) \Sigma(z) + \eqconstrdz(x, \abs{z}, z)) J_{z}(x, z).
    \end{aligned}
\end{equation}
The set of indices corresponding to active switches and the set of indices
corresponding to active inequality constraints are denoted by
\begin{equation}
    \actswitches(z) \coloneqq \set{i \in \oneto{\nswitches} \colon z_i = 0}
    \qquad\text{and}\qquad
    \actineq(x, z) \coloneqq \set{i \in \oneto{\nineqconstrs} \colon {\ineqconstr(x, \abs{z}, z)}_i = 0}.
\end{equation}
The projection onto the active switching indices is defined by \(\projactsw
\coloneqq \paren{e^T_i}_{i \in \actswitches(z)} \in
\reals^{\card{\actswitches(z)} \times \nswitches}\), where \(e_i \in
\reals^\nswitches\) is the \(i\)-th unit vector while the complementary
projection to the inactive indices is denoted by \(\projinactsw \coloneqq
\paren{e_i^T}_{i \in \oneto{\nswitches} \setminus \actswitches(z)} \in
\reals^{(\nswitches - \card{\actswitches(z)}) \times \nswitches}\).
By a slight abuse of notation, the projection onto the active/inactive
inequality constraints will similarly be denoted by \(\projactineq \coloneqq
(e^T_i)_{i \in \actineq(x, z)} \in \reals^{\card{\actineq(x, z)} \times
\nineqconstrs}\) and \(\projinactineq \coloneqq (e^T_i)_{i \in
\oneto{\nineqconstrs} \setminus \actineq(x, z)} \in \reals^{(\nineqconstrs -
\card{\actineq(x, z)}) \times \nineqconstrs}\).

\begin{definition}[LIKQ]\label{def:likq}
    Given \(\switchfct, \ineqconstr\) and \(\eqconstr\) as before, a point
    \(\tuple{x, z} \in \feasible(\switchfct, \ineqconstr, \eqconstr)\) is said
    to satisfy the \emph{linear independence kink qualification (LIKQ)} at
    \(\tuple{x, z}\), if
    \begin{equation}
        \rank
        \begin{bmatrix}
            \projactsw[\actswitches(z)] J_{z}(x, z)\\
            \projactineq[\actineq(x, z)] J_{\ineqconstr}(x, z)\\
            J_{\eqconstr}(x, z)
        \end{bmatrix}
        = \card{\actswitches(z)} + \card{\actineq(x, z)} + \neqconstrs.
    \end{equation}
\end{definition}

The examples below consider the simple case with just one switch and without
equality and inequality constraints.
\begin{example}[LIKQ everywhere]\label{ex:likq-everywhere-intro}
    Consider the switching function \(\switchfct(\optvar, \absvar, \switchvar)
    \coloneqq -\sin(\optvar)\) which has an active switch at \(x \in \set{\ell
    \pi \with \ell \in \ints}\).
    Since \(\switchfctdy(x, 0, 0) = \switchfctdz(x, 0, 0) = 0\) and
    \(\switchfctdx(x, 0, 0) = -\cos(x) \in \set{-1, 1}\) for all those points,
    the matrix \(\projactsw[\actswitches(z)]J_z(x, y, z)\) is either empty or
    \(\pm 1\), and hence, has full rank everywhere.
\end{example}
\begin{example}[no LIKQ at 0]\label{ex:no-likq-intro}
    For the second example consider \(\switchfct(\optvar, \absvar, \switchvar)
    \coloneqq \sin(\optvar) - \optvar\).
    At \(x = 0\) the switching equation yields \(z = \abs{z} = 0\) and the
    stratum containing \(\jet{0}(0, 0, 0)\) is \(\stratum[(0)]\).
    LIKQ does not hold at \((0, 0)\), since again \(\switchfctdy(0, 0, 0) =
    \switchfctdz(0, 0, 0) = 0\), and hence, \(J_{z}(x, z) = \switchfctdx(0, 0,
    0) = \cos(0) - 1 = 0\).
\end{example}
 
\section{Prerequisites from differential topology}
\label{sec:topo}
{
The arguments in the later sections are based on an extension of the
jet-transversality theorem, see~\cite[Thm.~7.4.5]{JJT00}, for structured jets
in a paper by H.\ Günzel~\cite{Gün08}.
For the sake of a complete and mostly self-consistent presentation the
important definitions of Günzels paper are restated here.

\newcommand{\fct}{\totalfct}
\newcommand{\var}{\totalvar}
\newcommand{\indim}{\totaldim}
\newcommand{\outdim}{\totalconstrs}
\newcommand{\point}{(\jval{\var}, \jval{\fct})}

To not overload the notation that was established in the first part the
notation here deviates in some points from the one that is used for example
in~\cite{JJT00} or~\cite{Gün08}.
This of course also prepares the application of the results in this section to
the abs-smooth problem class.
In the following the reader may think of \(\var \in \reals^\totaldim\) as a
tuple \(\tuple{\optvar, \absvar, \switchvar}\), \(\fct\) as the combined
function \((\switchfct, \ineqconstr, \eqconstr) \from \reals^\totaldim \to
\reals^\totalconstrs\), i.e., \(\fct(\var) = \fct(\optvar, \absvar, \switchvar)
= (\switchfct(\optvar, \absvar, \switchvar), \ineqconstr(\optvar, \absvar,
\switchvar), \eqconstr(\optvar, \absvar, \switchvar))\).
In contrast, \(\enlarged{\fct}\) can be thought of as the \((\nswitches+2)\)-tuple
of functions defining the abs-smooth problem in~\eqref{eq:abs-smooth-problem},
i.e, \(\enlarged{\fct} = (\switchfct_1, \ldots, \switchfct_\nswitches,
\ineqconstr, \eqconstr)\) and for \(i \in \oneto{s+2}\) the vector \(\var_{i}
\in \reals^{\totaldim_i}\) as a possible input to the \(i\)-th component of
\(\enlarged{\fct}\).
The vector \(\enlarged{\var}\) represents a vector of independent inputs to the
function \(\enlarged{\fct}\), i.e, \(\enlarged{\var} = ( \var_{1}, \ldots,
\totalvar_{\nswitches + 2})\).
Moreover, bold font is used to indicate variables that should be thought of as
values of the corresponding function, e.g., \(\jval{\fct} =
\fct(\var)\).

\begin{definition}[stratifications, jet-transversality~{\cite[Def.~7.3.33, Thm.~7.3.4]{JJT00}}]\label{def:stratifications-transversality}
Let \(\indim, \outdim \in \nats\) and \(\stratifiedset \subseteq
    \reals^{\indim+\outdim}\), then a locally finite partition
    \(\stratification\) of \(\stratifiedset\) into pairwise disjoint
    \(\contdiff\) manifolds is called a \emph{stratification} of
    \(\stratifiedset\) of class \(\contdiff\).
    The elements of \(\stratification\) are referred to as \emph{strata} and
    are indexed with some index set \(J\).
    The dimension of \(\stratification\) is defined by
     \begin{equation}
         \dim(\stratification)
         \coloneqq \max_{j \in J} \dim(\stratum[j]).
     \end{equation}

    For a given \(\point \in \stratifiedset\) the stratum containing \(\point\)
    is denoted by \(\stratum[\point]\) and the tangent space at \(\point\) on
    \(\stratum[\point]\) simply by \(\stratspace{\point} \coloneqq
    \tspace{\point}\stratum[\point]\).

    The stratification of \(\stratifiedset\) is called \emph{weakly Whitney
    regular} if for any \(\point \in \stratifiedset\) and any sequence
    \(\seq{\point_\ell}{\ell \in \nats}\) that converges to \(\point\) and
    consists of elements of the same stratum \(\stratum[j] \in
    \stratification\), for some \(j \in J\), the inclusion
    \begin{equation}
        \stratspace{\point} \subseteq \lim_{\ell \to \infty} \tspace{\point_\ell}{\stratum[j]}
    \end{equation}
    holds, whenever the last limit exists in the Grassmann manifold,
    cf.~\cite[Example~1.36]{Lee12}.

    For a function \(\fct \in \contdiff(\reals^\indim; \reals^\outdim)\) the
    \(0\)-jet-extension \(\jet{0}(\fct) \from \reals^\indim \to \reals^{\indim
    + \outdim}\) is given by \(\jet{0}(\fct)(\var) \coloneqq (\var,
    \fct(\var))\).
    The \enquote{\(0\)} in the term \(0\)-jet-extension refers to the fact that
    no derivative information of \(\fct\) is used in the jet-extension; a
    setting that suffices for the arguments in this paper.

The \(0\)-jet-extension \(\jet{0}(\fct)\) is said to meet
    \(\stratification\) \emph{transversally}, denoted
    \(\jet{0}(\fct)\transversal\stratification\), if for all \(\var \in
    {\jet{0}(\fct)}^{-1}(\stratifiedset)\)
    \begin{equation}
        \img(\D \jet{0}(\fct)(\var)) + \stratspace{\jet{0}(\fct)(\var)}
        = \reals^{\indim+\outdim}.
    \end{equation}
\end{definition}
Roughly speaking, the  classical jet-transversality theorem states that for an
open and dense set of functions \(\fct \in \contdiff(\reals^\indim;
\reals^\outdim)\) the 0-jet-extension of \(\fct\) meets any weakly Whitney
regular stratified set transversally.
Transversality theorems can be leveraged by encoding certain properties of the
inputs \(\var\), e.g., feasibility with respect to \(\fct\), into a stratified set
\(\stratifiedset\) and then reformulating the transversality condition into an algebraic
condition, e.g., a constraint qualification.
The transversality theorem then states on the one hand that the acquired
condition will hold true for all inputs \(\var\), possibly after an arbitrarily
small perturbation of \(\fct\) (density), and that it is stable (openness).

The more flexible version for structured jets deals with a finite collection of
functions and allows for the description of \(\stratifiedset\) to depend
differently on different variables.
For the remainder of this section let \(\nfcts \in \nats\) be the number of
involved functions and let \(\indim_1,\ldots,\indim_\nfcts\in \nats\) and
\(\outdim_1,\ldots,\outdim_\nfcts\in \nats\) denote the dimensions of their
respective domain and image spaces.
Furthermore, let \(\enlarged{\indim} \in \nats\) be the total amount of
independent inputs to the \(\nfcts\) functions and \(\outdim \in \nats\)
the total number of their values, i.e,
\begin{equation}
    \enlarged{\indim} \coloneqq \indim_1 + \ldots + \indim_\nfcts
    \quad\text{and}\quad
    \outdim \coloneqq \outdim_1 + \ldots + \outdim_\nfcts.
\end{equation}

\begin{definition}[structured 0-jet-extensions~{\cite[Definition~2.4]{Gün08}}]\label{def:jet-extensions}
    For \(i \in \oneto{\nfcts}\) let \(\fct_i \in \contdiff(\reals^{\indim_i};
    \reals^{\outdim_i})\) and \(\enlarged{\fct} \coloneqq (\fct_1, \ldots,
    \fct_\nfcts) \), then the function
\begin{align}
        \strjet{0}(\enlarged{\fct})
        \colon \reals^{\indim_1} \times \ldots \times \reals^{\indim_\nfcts}
        \to \reals^{\indim_1} \times \ldots \times \reals^{\indim_\nfcts} \times \reals^{\outdim_1} \times \ldots \times \reals^{\outdim_\nfcts}
        \shortintertext{given by}
        \strjet{0}(\enlarged{\fct})(\var_{1}, \ldots, \var_{\nfcts})
        \coloneqq (\var_{1}, \ldots, \var_{\nfcts}, \fct_1(\var_{1}), \ldots, \fct_\nfcts(\var_{\nfcts}))
    \end{align}
    is called \emph{structured 0-jet-extension} of \(\enlarged{\fct}\).
\end{definition}
The tuple of functions \(\enlarged{\fct}\) in the above definition can be
understood as an overall function in \(\contdiff(\reals^{\enlarged{\indim}};
\reals^\outdim)\) by defining its value for \(\enlarged{\var} \coloneqq
(\var_{1}, \ldots, \var_{\nfcts}) \in \reals^{\indim_1} \times \ldots \times
\reals^{\indim_\nfcts} \isom \reals^{\enlarged{\indim}}\) as
\begin{equation}\label{eq:def-structured-eval}
    \enlarged{\fct}(\enlarged{\var}) \coloneqq (\fct_1(\var_{1}), \ldots, \fct_\nfcts(\var_{\nfcts})).
\end{equation}
Note that this evaluation of \(\enlarged{\fct}\) allows to view the
space \(\contdiff(\reals^{\indim_1}; \reals^{\outdim_1}) \times \ldots
\times \contdiff(\reals^{\indim_\nfcts}; \reals^{\outdim_\nfcts})\) as a
subspace of \(\contdiff(\reals^{\enlarged{\indim}}; \reals^\outdim)\).
In fact, the values of the structured jet \(\strjet{0}(\enlarged{\fct})\) are
identical to the values of the classical jet
\(\enlarged{\var} \mapsto (\enlarged{\var}, \enlarged{\fct}(\enlarged{\var}))\)
when using this \emph{structured} way of evaluating \(\enlarged{\fct}\).
This justifies using the same notation for structured jets as for classical
jets.
Moreover, the notation of jet-transversality from
Definition~\ref{def:stratifications-transversality} extends naturally to
structured jets.
However, since the input \(\var_{i}\) does not affect the output of \(\fct_j\)
whenever \(j \ne i\), the previously described subspace relation is strict.
Specifically, the possible perturbations of \(\enlarged{\fct}\) viewed as an
element in \(\contdiff(\reals^{\indim_1}; \reals^{\outdim_1}) \times \ldots
\times \contdiff(\reals^{\indim_\nfcts}; \reals^{\outdim_\nfcts})\) are
different from the possible perturbations of \(\enlarged{\fct}\) as an element
of \(\contdiff(\reals^{\enlarged{\indim}}; \reals^\outdim)\).
The topology in which the perturbations have to be understood is clarified by
the following definition.
\begin{definition}[(strong) Whitney topology]
    \newcommand{\otherfct}{\psi}
    Let \(\ell \in \zeroto{\smoothness}\), \(i \in \oneto{\nfcts}\).
    The sets \(U_{\fct_i, \varepsilon}\), which are indexed by continuous functions
    \(\varepsilon \from \reals^{\indim_i} \to (0, \infty)\)
    and points
    \(\fct_i \in \contdiff(\reals^{\indim_i}; \reals^{\outdim_i})\),
    and which are defined by
    \begin{equation}
        U_{\fct_i, \varepsilon}
        \coloneqq
        \set*{\otherfct \in \contdiff(\reals^{\indim_i}; \reals^{\outdim_i})
            \with
            \begin{aligned}
                \norm{\partial_{\multidx}\fct(\var_{i}) - \partial_{\multidx}\otherfct(\var_{i})} &< \varepsilon(\var_{i})\\
                \text{ for all } \var_{i} \in \reals^{\indim_i}, \card{\multidx} &\le \ell
            \end{aligned}
        },
    \end{equation}
form a basis of neighborhoods of
    the \emph{(strong) Whitney \(\contdiff[\ell]\)-topology} of
    \(\contdiff(\reals^{\indim_i}; \reals^{\outdim_i})\).
\end{definition}
If in the definition one uses only positive constants \(\varepsilon\) instead
of positive functions one obtains a basis of the coarser weak Whitney topology.
However, since \(\reals^{\indim_i}\) is not compact, the convergence with respect
to this simpler topology does not capture the behavior \enquote{at infinity}
very well.
In contrast, the strong topology is an extremely fine topology and in fact is
not metrizable~\cite{Kri69}.
Nevertheless, it forms a Baire space, which allows to meaningfully define
\emph{generic} subsets as sets which contain the intersection of a countable
number of open and dense subsets.
In particular, open and dense sets themselves are generic.
For details see, e.g.,~\cite[p.~34--36, Theorem~4.4.]{Hir76}
and~\cite[p.~306]{JJT00}.

For \(\ell \le \smoothness\) the product topology on
\(\contdiff(\reals^{\indim_1}; \reals^{\outdim_1}) \times \ldots \times
\contdiff(\reals^{\indim_\nfcts}; \reals^{\outdim_\nfcts})\) induced by the
Whitney \(\contdiff[\ell]\)-topologies on the respective spaces will be
referred to as the Whitney \(\contdiff[\ell]\)-topology of
\(\contdiff(\reals^{\indim_1}; \reals^{\outdim_1}) \times \ldots \times
\contdiff(\reals^{\indim_\nfcts}; \reals^{\outdim_\nfcts})\).
This completes the prerequisites required to state the structured
jet-transversality theorem.
\begin{theorem}[structured jet-transversality~{\cite[Theorem~2.5]{Gün08}}]\label{thm:structured-jet-transversality}
    Let \(\largestratification\) be a weakly Whitney regular stratification of
    \(\largestratifiedset \subseteq \reals^{\indim_1} \times \reals^{\outdim_1} \times \ldots
    \times \reals^{\indim_\nfcts} \times \reals^{\outdim_\nfcts}\)with \(\dim(\largestratification) < k + \outdim\).
Then,
    \begin{equation}
        \goodconstrs
        \coloneqq
        \set{\enlarged{\fct} \in \contdiff(\reals^{\indim_1}; \reals^{\outdim_1}) \times \ldots \times \contdiff(\reals^{\indim_\nfcts}; \reals^{\outdim_\nfcts})
            \with
            \strjet{0}(\enlarged{\fct}) \transversal \largestratification
        }
    \end{equation}
    is a dense subset of \(\contdiff(\reals^{\indim_1}; \reals^{\outdim_1}) \times
    \ldots \times \contdiff(\reals^{\indim_\nfcts}; \reals^{\outdim_\nfcts})\) with
    respect to the Whitney \(\contdiff\)-topology.
    If \(\largestratifiedset\) is closed, then \(\goodconstrs\) is additionally open
    with respect to the Whitney \(\contdiff[1]\)-topology, and hence, certainly
    open with respect to the Whitney \(\contdiff\)-topology.
\end{theorem}

As already stated, the concept of a structured jet and the structured
jet-transversality theorem are useful to restrict the perturbations considered
in the desired space of functions.
However, from a technical point of view that comes with the downside of having
independent variables for every function that is part of the tuple \(\fct\).
Mathematically this is not a problem as relations between these variables can
also be expressed in the stratified set \(\stratifiedset\), see, e.g.,
\cite[Example~3.2]{Gün08}.
The following lemma represents a tool to alleviate this problem by means of a
left-invertible linear operator \(\Pi\) that relates between a few actual
variables and the many formally required variables of the structured jet.
The proof is rather elementary and similar, although not equivalent statements
can be found as exercises in the book~\cite{JJT00}.
Nevertheless, the proof is presented here also with the intent to provide a
reference to this tool for future works.
\begin{lemma}\label{lem:simple-jets}
    Let \(\Pi \in \reals^{\enlarged{\indim} \times \indim}\) be a
    left-invertible matrix and \(J\) an index set for a stratification
    \(\stratification \coloneqq \family{\stratum[j]}_{j \in J}\) of
    \(\stratifiedset \subseteq \reals^{\indim + \outdim}\) of class
    \(\contdiff\).
    Additionally, for \(i \in \oneto{\nfcts}\) let \(\fct_i \in
    \contdiff(\reals^{\indim_i}; \reals^{\outdim_i})\) and \(\enlarged{\fct}
    \coloneqq (\fct_1, \ldots, \fct_\nfcts)\).
Then,
    \begin{enumerate}
        \item\label{it:is-stratification}
            \(\largestratification \coloneqq
            \family{\largestratum[j]}_{j\in J}\) is a stratification of \(\largestratifiedset
            \subseteq \reals^{\enlarged{\indim} + \outdim}\) into \(\contdiff\) manifolds, where
            \begin{equation}
                \largestratum[j] \coloneqq \begin{bmatrix} \Pi & 0 \\ 0 & \Id{\outdim} \end{bmatrix} \stratum[j]
                \qquad\text{and}\qquad
                \tilde\stratifiedset = \begin{bmatrix} \Pi & 0 \\ 0 & \Id{\outdim} \end{bmatrix} \stratifiedset.
            \end{equation}\item\label{it:is-Whitney-regular}
            if \(\stratification\) is weakly Whitney regular, so is \(\largestratification\).

        \item\label{it:is-transversal}
            \(\strjet{0}(\enlarged{\fct}) \transversal \largestratification\)
            if and only if
            \(\jet{0}(\fct) \transversal \stratification\), where \(\fct
            \coloneqq \enlarged{\fct} \after \Pi\)
            using~\eqref{eq:def-structured-eval}.
    \end{enumerate}
\end{lemma}

\begin{proof}
    \begin{description}
        \item[\ref{it:is-stratification}]
            Clearly, \(\largestratification\) is a partition of
            \(\largestratifiedset\).
            For any \(j \in J\) the mapping
            \begin{equation}
                \reals^{\indim+\outdim} \supset \stratum[j] \to \reals^{\enlarged{\indim}+\outdim},
                \qquad
                (\jval{\var}, \jval{\fct})
                \mapsto
                (\enlarged{\jval{\var}}, \enlarged{\jval{\fct}})
                \coloneqq \begin{bmatrix} \Pi & 0 \\ 0 & \Id{\outdim} \end{bmatrix} (\jval{\var}, \jval{\fct})
            \end{equation}
            is an injective smooth immersion.
            Thus, \cite[Proposition~5.18]{Lee12} ensures that
            \(\largestratum[j]\) is a \(\contdiff\) manifold.
            To show that \(\largestratification\) is locally finite, let some
            \((\enlarged{\jval{\var}}, \enlarged{\jval{\fct}}) \in
            \reals^{\enlarged{\indim} + \outdim}\) be given.
            If \(\enlarged{\jval{\var}} \notin \img \Pi\), then there is a
            whole neighborhood \(U_{\enlarged{\jval{\var}}} \subset
            \reals^{\enlarged{\indim}}\) of \(\enlarged{\jval{\var}}\) with
            \(U_{\enlarged{\jval{\var}}} \cap \img \Pi = \emptyset\) due to
            \(\img \Pi\) being a closed subspace of
            \(\reals^{\enlarged{\indim}}\).
            Then, for any neighborhood \(U_{\enlarged{\jval{\fct}}} \subset
            \reals^\outdim\) of \(\enlarged{\jval{\fct}}\), clearly
            \((U_{\enlarged{\jval{\var}}} \times U_{\enlarged{\jval{\fct}}}) \cap
            \largestratum[j] = \emptyset\) for any \(j \in J\), since
            \(\largestratum[j] \subseteq \img \Pi \times \reals^\outdim\).

            If otherwise \(\enlarged{\jval{\var}} \in \img \Pi\), then there is
            \(\jval{\var} \in \reals^\indim\) with \(\Pi \jval{\var} =
            \enlarged{\jval{\var}}\).
            Since \(\stratification\) is locally finite there is a neighborhood
            \(U_{(\jval{\var}, \jval{\fct})} \subset \reals^{\indim
            + \outdim}\) of \((\jval{\var}, \enlarged{\jval{\fct}})\) such that
            \(\set{j \in J \colon U_{(\jval{\var}, \enlarged{\jval{\fct}})}
            \cap \stratum[j] \neq \emptyset}\) is finite.
            Set
            \begin{equation}
                U_{(\enlarged{\jval{\var}}, \enlarged{\jval{\fct}})}
                \coloneqq \begin{bmatrix} \Pi & 0 \\ 0 & \Id{\outdim} \end{bmatrix} U_{(\jval{\var}, \enlarged{\jval{\fct}})},
            \end{equation}
            then, for any \(j \in J\) with \(U_{(\jval{\var},
            \enlarged{\jval{\fct}})} \cap \stratum[j] = \emptyset\), it holds
            \begin{equation}
                U_{(\enlarged{\jval{\var}}, \enlarged{\jval{\fct}})} \cap \largestratum[j]
                = \begin{bmatrix} \Pi & 0 \\ 0 & \Id{\outdim} \end{bmatrix} U_{(\jval{\var}, \enlarged{\jval{\fct}})}
                \cap
                \begin{bmatrix} \Pi & 0 \\ 0 & \Id{\outdim} \end{bmatrix} \stratum[j]
                = \begin{bmatrix} \Pi & 0 \\ 0 & \Id{\outdim} \end{bmatrix} (U_{(\jval{\var}, \enlarged{\jval{\fct}})} \cap \stratum[j])
                = \emptyset,
            \end{equation}
            where the penultimate equality holds due to \(\Pi\) being
            injective.
            Therefore, the set \(\set{j \in J \colon
            U_{(\enlarged{\jval{\var}}, \enlarged{\jval{\fct}})} \cap
        \largestratum[j] \neq \emptyset}\) is finite.

        \item[\ref{it:is-Whitney-regular}]
            Let \( (\jval{\var}, \jval{\fct})\in \stratifiedset\), then
            \begin{equation}
                \label{eq:transform-tangent-space}
                \largestratspace{(\Pi \jval{\var},\jval{\fct})}
                =
                \begin{bmatrix} \Pi & 0 \\ 0 & \Id{\outdim} \end{bmatrix}
                \stratspace{(\jval{\var}, \jval{\fct})}.
            \end{equation}
            Let \(\seq{\enlarged{\jval{\var}}_\ell,
            \enlarged{\jval{\fct}}_\ell}{\ell\in\nats} \subset
            \largestratum[j]\) for some \(j \in J\) with
            \((\enlarged{\jval{\var}}_\ell, \enlarged{\jval{\fct}}_\ell) \to
            (\enlarged{\jval{\var}}, \enlarged{\jval{\fct}})\) as \(\ell \to
            \infty\) and let \(\tilde{T}\) be the limit of
            \(\tspace{(\enlarged{\jval{\var}}_\ell,
            \enlarged{\jval{\fct}}_\ell)}{\largestratum[j]}\) in the
            corresponding Grassmannian.
            Since, for every \(\ell \in \nats\),
            \((\enlarged{\jval{\var}}_\ell, \enlarged{\jval{\fct}}_\ell) \in
            \largestratum[j]\), in particular \(\enlarged{\jval{\var}}_\ell \in
            \img \Pi\), there is \(\jval{\var}_\ell \in \reals^\indim\) with
            \(\enlarged{\jval{\var}}_\ell = \Pi \jval{\var}_\ell\) and
            \((\jval{\var}_\ell, \enlarged{\jval{\fct}}_\ell) \in \stratum[j]\).
            It holds \((\jval{\var}_\ell, \enlarged{\jval{\fct}}_\ell) \to (\Pi^\pinv
            \enlarged{\jval{\var}}, \enlarged{\jval{\fct}})\) as \(\ell \to \infty\) and
            \begin{equation}
                T \coloneqq
                \begin{bmatrix}
                    \Pi^\pinv & 0\\
                    0 & \Id{\outdim}
                \end{bmatrix}
                \tilde{T}
                =
                \begin{bmatrix}
                    \Pi^\pinv & 0\\
                    0 & \Id{\outdim}
                \end{bmatrix}
                \lim_{\ell \to\infty} \largestratspace{(\enlarged{\jval{\var}}_\ell, \enlarged{\jval{\fct}}_\ell)}
                = \lim_{\ell \to \infty} \largestratspace{(\jval{\var}_\ell, \enlarged{\jval{\fct}}_\ell)}.
            \end{equation}
            Since \(\stratification\) is weakly Whitney regular, it follows
            \(\stratspace{(\Pi^\pinv \enlarged{\jval{\var}}, \enlarged{\jval{\fct}})}
            \subseteq T\).
            Thereby,
            \begin{equation}
                \largestratspace{(\enlarged{\jval{\var}}, \enlarged{\jval{\fct}})}
                =
                \begin{bmatrix}
                    \Pi & 0\\
                    0 & \Id{\outdim}
                \end{bmatrix}
                \stratspace{(\Pi^\pinv \enlarged{\jval{\var}}, \enlarged{\jval{\fct}})}
                \subseteq
                \begin{bmatrix}
                    \Pi & 0\\
                    0 & \Id{\outdim}
                \end{bmatrix} T
                = \tilde{T}.
            \end{equation}

        \item[\ref{it:is-transversal}]
            A vector \(\enlarged{\var}\) is an element of
            \({\jet{0}(\enlarged{\fct})}^{-1}\paren{\largestratifiedset}\) if
            and only if \((\enlarged{\var}, \enlarged{\fct}(\enlarged{\var}))
            \in \largestratifiedset\).
            The latter is the case if and only if \(\enlarged{\var} \in \img
            \Pi\) and \((\Pi^\pinv \enlarged{\var},
            \enlarged{\fct}(\enlarged{\var})) \in \stratifiedset\).
            Since for any \(\enlarged{\var} \in \img \Pi\), it holds
            that
            \begin{equation}
                (\Pi^\pinv \enlarged{\var}, \enlarged{\fct}(\enlarged{\var}))
                = (\Pi^\pinv \enlarged{\var}, \enlarged{\fct}(\Pi\, \Pi^\pinv \enlarged{\var}))
                = \jet{0}(\enlarged{\fct} \after \Pi)(\Pi^\pinv \enlarged{\var})
                = \jet{0}(\fct)(\Pi^\pinv \enlarged{\var}),
            \end{equation}
            the former conditions are all equivalent to
            \(\enlarged{\var} \in \img \Pi \) and
            \(\jet{0}(\fct)(\Pi^\pinv
            \enlarged{\var}) \in \stratifiedset\).
            Finally, this is equivalent to \(\enlarged{\var} \in \Pi
            {\jet{0}(\fct)}^{-1}(\stratifiedset)\), which
            shows that
            \begin{equation}\label{eq:relation-of-preimages}
                {\jet{0}(\enlarged{\fct})}^{-1}\paren{\largestratifiedset}
                = \Pi {\jet{0}(\fct)}^{-1}(\stratifiedset).
            \end{equation}

            \newcommand{\tdim}{t}
            \newcommand{\fbm}{B_{\totaldim}}
            \newcommand{\sbm}{B_{\outdim}}
            Now fix \(\var \in {\jet{0}(\fct)}^{-1}(\stratifiedset)\)
            and let \(\tdim \coloneqq
            \dim(\stratspace{\jet{0}(\fct)(\var)})\) and \(\fbm \in
            \reals^{\indim \times \tdim}\) and \(\sbm \in \reals^{\outdim
            \times \tdim}\) be such that the columns of
            \begin{equation}
                \begin{bmatrix}
                    \fbm \\ \sbm
                \end{bmatrix} \in \reals^{(\indim + \outdim) \times \tdim}
            \end{equation}
            form a basis of \(\stratspace{\jet{0}(\fct)(\var)}\) and the
            condition \(\reals^{\indim + \outdim} = \img(\D
            \jet{0}(\fct)(\var)) +
            \stratspace{\jet{0}(\fct)(\var)}\) is equivalent to
            \begin{equation}\label{eq:rank-second}
                \rank \begin{bmatrix}
                    \Id{\indim} & \fbm\\
                    \D \fct(\var) & \sbm
                \end{bmatrix}
                =
                \rank \begin{bmatrix}
                    \Id{\indim} & \fbm\\
                    \D \enlarged{\fct} (\Pi \var)\Pi & \sbm
                \end{bmatrix}
                = \indim + \outdim.
            \end{equation}
            Let further \(\Lambda \in \reals^{{\enlarged{\indim}} \times
            ({\enlarged{\indim}}-\indim)}\) be left-invertible such that the
            columns of \(\Lambda\) span the
            \(({\enlarged{\indim}}-\indim)\)-dimensional orthogonal complement
            of \(\img \Pi \), i.e., \(\img \Lambda  = {(\img \Pi)}^\bot\).
            Then, \(\Pi^T \Lambda = 0\) and \([\Lambda\
            \Pi]\in\reals^{{\enlarged{\indim}} \times {\enlarged{\indim}}}\)
            has full rank.
            Using this split,~\eqref{eq:rank-second} is equivalent to
            \begin{equation}
                \rank
                \begin{bmatrix}
                    \Id{{\enlarged{\indim}}-\indim} & 0 & 0\\
                    0 & \Id{\indim} & \fbm \\
                    \D \enlarged{\fct}(\Pi \var)\Lambda & \D \enlarged{\fct}(\Pi \var) \Pi & \sbm
                \end{bmatrix}
                = ({\enlarged{\indim}}-\indim) + \indim + \outdim
                = {\enlarged{\indim}} + \outdim.
            \end{equation}
            Moreover, since
            \begin{equation}
                \begin{bmatrix}
                    \Id{{\enlarged{\indim}}-\indim} & 0 & 0\\
                    0 & \Id{\indim} & \fbm \\
                    \D \enlarged{\fct}(\Pi \var)\Lambda & \D \enlarged{\fct}(\Pi \var) \Pi & \sbm
                \end{bmatrix}
                =
                \begin{bmatrix}
                    \Lambda^\dagger & 0\\
                    \Pi^\dagger & 0 \\
                    0 & \Id{\outdim}
                \end{bmatrix}\!
                \begin{bmatrix}
                    \Id{\indim} & \Pi \fbm\\
                    \D \enlarged{\fct}(\Pi \var) & \phantom{\Pi}\sbm
                \end{bmatrix}\!
                \begin{bmatrix}
                    \Lambda & \Pi & 0\\
                    0 & 0 & \Id{\outdim}
                \end{bmatrix}\!
            \end{equation}
            and the transformation matrices have full rank
            \({\enlarged{\indim}} + \outdim\) the above is equivalent to
            \begin{equation}
                \rank
                \begin{bmatrix}
                    \Id{\indim} & \Pi \fbm\\
                    \D \enlarged{\fct}(\Pi \var) & \phantom{\Pi}\sbm
                \end{bmatrix}
                = {\enlarged{\indim}} + \outdim
            \end{equation}
            which is equivalent to
            \begin{equation}
                \reals^{{\enlarged{\indim}} + \outdim}
                = \img(\D \jet{0}(\enlarged{\fct})(\Pi \var)) + \begin{bmatrix} \Pi & 0 \\ 0 & \Id{\outdim} \end{bmatrix}\stratspace{\jet{0}(\enlarged{\fct} \after \Pi)(\var)}.
            \end{equation}
            The equivalence of the transversality conditions then follows by
            \begin{align}
                &\phantom{\iff} \jet{0}(\enlarged{\fct}) \transversal \largestratification \\
                &\stackrel{\text{\eqref{eq:relation-of-preimages}}}{\iff}
                \forall\, \enlarged{\var} \in \Pi {\jet{0}(\fct)}^{-1}(\stratifiedset) \colon
                \reals^{{\enlarged{\indim}} + \outdim} = \img(\D \jet{0}(\enlarged{\fct})(\enlarged{\var})) + \largestratspace{\jet{0}(\enlarged{\fct})(\enlarged{\var})}\\
&\stackrel{\text{\eqref{eq:transform-tangent-space}}}{\iff}
                \forall\, \var \in {\jet{0}(\fct)}^{-1}(\stratifiedset) \colon
                \reals^{{\enlarged{\indim}} + \outdim} = \img(\D \jet{0}(\enlarged{\fct})(\Pi \var)) + \begin{bmatrix} \Pi & 0 \\ 0 & \Id{\outdim} \end{bmatrix}\stratspace{\jet{0}(\enlarged{\fct} \after \Pi)(\var)}\\
&\iff
                \forall\, \var \in {\jet{0}(\fct)}^{-1}(\stratifiedset) \colon
                \reals^{\indim + \outdim} = \img(\D \jet{0}(\fct)(\var)) + \stratspace{\jet{0}(\fct)(\var)}\\
&\stackrel{\text{def.}}{\iff}
                \jet{0}(\fct) \transversal \stratification.
                \qedhere \end{align}
    \end{description}
\end{proof}
}
 
\section{LIKQ as a transversality condition}\label{sec:likq-as-transversality}
For given functions  \(\switchfct \in \switchfcts\), \(\ineqconstr \in
\ineqconstrs\) and \(\eqconstr \in \eqconstrs\), this section establishes a
transversality condition for the remarkable property that all feasible points
\(\tuple{\optvar, \switchvar} \in \feasible(\switchfct, \ineqconstr,
\eqconstr)\) satisfy the LIKQ condition.
To simplify notation set \(\totalconstrs \coloneqq \nswitches + \nineqconstrs +
\neqconstrs\) and let \(\jet{0}(\switchfct, \ineqconstr, \eqconstr)\) be the
standard \(0\)-jet for the combined function \(\tuple{\switchfct, \ineqconstr,
\eqconstr} \from \reals^\totaldim \to \reals^\totalconstrs\), i.e.,
\begin{equation}\label{eq:def-jet}
    \jet{0}(\switchfct, \ineqconstr, \eqconstr)(x, y, z) \coloneqq {(x, y, z, \switchfct(x, y, z), \ineqconstr(x, y, z), \eqconstr(x, y, z))}.
\end{equation}
The definition of the set
\(\stratifiedset \subseteq \reals^{\totaldim + \totalconstrs}\)
as
\begin{equation}\label{eq:def-stratified-set}
    \stratifiedset
    \coloneqq
    \set{\tuple{\jval{x}, \jval{y}, \jval{z}, \jval{\switchfct}, \jval{\ineqconstr}, \jval{\eqconstr}}
         \colon
             \jval{y} = \abs{\jval{z}}, \jval{z} = \jval{\switchfct},
             \jval{\ineqconstr} \ge 0, \jval{\eqconstr} = 0
     }
\end{equation}
then ensures that \(\tuple{x, z} \in \feasible(\switchfct, \ineqconstr,
\eqconstr)\) is equivalent to \(\tuple{x, \abs{z}, z} \in {\jet{0}(\switchfct,
\ineqconstr, \eqconstr)}^{-1}(\stratifiedset)\); which is a fundamental
prerequisite to the subsequent analysis.
A stratification of \(\stratifiedset\) is given by \(\stratification \coloneqq
\set{\stratum[\sigma, \omega] \colon \sigma \in \set{-1, 0, 1}^\nswitches,
\omega\in\set{0, 1}^\nineqconstrs}\),
i.e.,
\begin{equation}
    \stratifiedset
    = \Union_{{\substack{\sigma \in \set{-1, 0, 1}^\nswitches \\[.2em] \omega \in \set{0, 1}^\nineqconstrs}}} \stratum[\sigma, \omega]
    \qquad\text{with}\qquad
    \stratum[\sigma, \omega] \coloneqq  \reals^\xdim \times \stratum[\sigma] \times \stratum[\omega] \times \set{0}^\neqconstrs
\end{equation}
and
\begin{equation}\label{eq:def-strata}
    \begin{aligned}
    \stratum[\sigma]
    & \coloneqq \set*{\tuple{\jval{y}, \jval{z}, \jval{\switchfct}} \in \reals^\nswitches \times \reals^\nswitches \times \reals^\nswitches
                   \colon
                   \jval{y} = \abs{\jval{z}}, \jval{z} = \jval{\switchfct}, \sign(\jval{z}) = \sigma
               },
    \\
    \stratum[\omega]
    & \coloneqq \set{\jval{\ineqconstr} \in \reals^\nineqconstrs
                   \colon
                   \sign(\jval{\ineqconstr}) = \omega
               }.
    \end{aligned}
\end{equation}
To analyze the transversality condition \(\jet{0}(\switchfct, \ineqconstr,
\eqconstr) \transversal \stratifiedset\) it is useful to characterize the
tangent spaces of the strata as images.
This is achieved by the following lemma.
\begin{lemma}[Tangent spaces and Whitney regularity]\label{lem:statum-tspace}
    Define for fixed  \(\sigma \in  \set{-1, 0, 1}^\nswitches\) and
    \(\omega \in \set{0, 1}^\nineqconstrs\) the matrices
    \(\Sigma \coloneqq \diag(\sigma)\) and
    \(\Omega \coloneqq \diag(\omega)\).
    Then, for \(\tuple{\jval{y}, \jval{z}, \jval{\switchfct}} \in
    \stratum[\sigma]\), and \(\jval{g} \in \stratum[\omega]\), the tangent
    spaces are given by
    \begin{equation}\label{eq:tspace-stratum-sigma}
        \tspace{\tuple{\jval{y}, \jval{z}, \jval{\switchfct}}}{\stratum[\sigma]}
         = \img \begin{bmatrix} \abs{\Sigma} \\ \Sigma \\ \Sigma \end{bmatrix}
    \end{equation}
    and \(\tspace{\jval{g}}{\stratum[\omega]} = \img \Omega\).
    Moreover, the stratification \(\stratification\) is weakly Whitney regular.
\end{lemma}
\begin{proof}
    \newcommand{\curve}{\gamma}
    \newcommand{\tvector}{\nu}
    \newcommand{\tvectorpreimage}{\eta}
    Let \(\tvector \in \tspace{\tuple{\jval{y}, \jval{z},
    \jval{\switchfct}}}{\stratum[\sigma]}\), then there is \(\varepsilon > 0\)
    and a differentiable curve \(\curve \from (-\varepsilon, \varepsilon) \to
    \stratum[\sigma]\) with \(\curve(0) = \tuple{\jval{y}, \jval{z},
    \jval{\switchfct}}\) and \(\curve'(0) = \tvector\).
The definition of \(\stratum[\sigma]\)
    implies
    \begin{equation}
        \curve(t)
        \eqqcolon \begin{bmatrix} \curve_y(t) \\ \curve_z(t) \\ \curve_\switchfct(t) \end{bmatrix}
        = \begin{bmatrix} \abs{\curve_z(t)} \\ \curve_z(t) \\ \curve_z(t) \end{bmatrix}
        = \begin{bmatrix} \Id{\nswitches} \\ \diag(\sigma) \\ \diag(\sigma) \end{bmatrix} \abs{\curve_z(t)}
= \begin{bmatrix} \abs{\Sigma} \\ \Sigma \\ \Sigma \end{bmatrix} \Sigma \curve_z(t),
    \end{equation}
    and hence,
    \begin{equation}
        \tvector
        = \curve'(0)
        = \begin{bmatrix} \abs{\Sigma} \\ \Sigma \\ \Sigma \end{bmatrix} \Sigma \curve'_z(0)
        \in \img \begin{bmatrix} \abs{\Sigma} \\ \Sigma \\ \Sigma \end{bmatrix}.
    \end{equation}
For \(\tvector\) in the right-hand side of~\eqref{eq:tspace-stratum-sigma}
    there is \(\tvectorpreimage \in \reals^\nswitches\) such that
    \begin{equation}
        \curve(t)
        \coloneqq \begin{bmatrix} \curve_y(t) \\ \curve_z(t) \\ \curve_\switchfct(t) \end{bmatrix}
        \coloneqq \begin{bmatrix} \jval{y} \\ \jval{z} \\ \jval{\switchfct} \end{bmatrix} + t \tvector
= \begin{bmatrix} \abs{\jval{z}} \\ \jval{z} \\ \jval{z} \end{bmatrix} + t \begin{bmatrix} \abs{\Sigma} \\ \Sigma \\ \Sigma \end{bmatrix} \tvectorpreimage.
    \end{equation}
    The definition of \(\curve\) ensures \(\curve(0) = (\jval{y}, \jval{z},
    \jval{\switchfct})\) and \(\curve'(0) = \tvector\). Therefore, it remains to show
    that in the vicinity of the origin \(\curve\) is a curve on
    \(\stratum[\sigma]\). One has \(\curve_z = \curve_\switchfct\).
    Since for \(i \in \oneto{s}\) the term \((t \Sigma \tvectorpreimage)_i\)
    vanishes whenever \(\sigma_i = 0\), there is a radius
    \(\varepsilon > 0\) such that for \(t < \varepsilon\)
    \begin{equation}
        \sign(\curve_z(t))
        = \sign(\jval{z} + t \Sigma \tvectorpreimage)
        = \sign(\jval{z})
        = \sigma.
    \end{equation}
    In particular,
    \(
        \abs{\curve_z(t)}
        = \Sigma \curve_z(t)
        = \Sigma \jval{z} + t \Sigma^2 \tvectorpreimage
        = \abs{\jval{z}} + t \abs{\Sigma} \tvectorpreimage
        = \curve_y(t)
    \),
    whence, \(\curve(t) \in \stratum[\sigma]\).
    The assertion for the tangent space of \(\stratum[\omega]\) follows
    analogously.

    \newcommand{\totalconstr}{\phi}
    \newcommand{\limit}[1]{\mathring{#1}}
    To show that \(\stratification\) is weakly Whitney regular, let
    \(\ell \in \nats\) and \(\jval{\totalvar}_\ell \coloneqq
    (\jval{\optvar}_\ell, \jval{\absvar}_\ell, \jval{\switchvar}_\ell)\),
    \(\jval{\totalconstr}_\ell \coloneqq (\jval{\switchfct}_\ell,
    \jval{\ineqconstr}_\ell, \jval{\eqconstr}_\ell)\) such that
    \((\jval{\totalvar}_\ell, \jval{\totalconstr}_\ell)\) is a sequence in
    \(\stratum[\sigma, \omega]\) that converges to some limit point
    \((\limit{\jval{\totalvar}}, \limit{\jval{\totalconstr}}) \coloneqq
    (\limit{\jval{x}}, \limit{\jval{y}}, \limit{\jval{z}},
    \limit{\jval{\switchfct}}, \limit{\jval{\ineqconstr}},
    \limit{\jval{\eqconstr}}) \in \stratifiedset\).
    Let \(\limit{\sigma} \coloneqq \sign(\limit{\jval{z}})\), \(\limit{\Sigma}
    \coloneqq \diag(\limit{\sigma})\), \(\limit{\omega} \coloneqq
    \sign(\limit{\jval\ineqconstr})\) and \(\limit{\Omega} \coloneqq
    \diag(\limit{\omega})\).
    Since \((\limit{\jval{y}}, \limit{\jval{z}}, \limit{\jval{\switchfct}})
    \in \stratum[\limit{\sigma}]\) and \(\limit{\jval{\ineqconstr}} \in
    \stratum[\limit{\omega}]\), it suffices to relate \(\sigma\) to
    \(\limit{\sigma}\), and \(\omega\) to \(\limit{\omega}\).
    Now assume there is \(i \in \oneto{\nswitches}\) with \(\sigma_i =
    -\limit{\sigma}_i \ne 0\), then there is a neighborhood of
    \(\limit{\jval{z}}\) in which the same relation holds true, which
    contradicts \(\jval{z}_\ell \to \limit{\jval{z}}\), as \(\sigma =
    \sign(\jval{z}_\ell)\) for all \(\ell \in \nats\).
    Therefore, \(\limit{\sigma}_i = \sigma_i\) or \(\limit{\sigma}_i = 0\) for
    all \(i \in \oneto{\nswitches}\).
    The same argument shows \(\limit{\omega}_i = \omega_i\) or
    \(\limit{\omega}_i = 0\).
    The previously established characterizations of the tangent spaces then yield
    \begin{equation}
        \stratspace{(\jval{\totalvar}, \jval{\totalconstr})}
        =
        \img\begin{bmatrix}
            \Id{\xdim} & 0                    & 0              & 0 \\
            0          & \abs{\limit{\Sigma}} & 0              & 0 \\
            0          & \limit{\Sigma}       & 0              & 0 \\
            0          & \limit{\Sigma}       & 0              & 0 \\
            0          & 0                    & \limit{\Omega} & 0 \\
            0          & 0                    & 0              & 0
        \end{bmatrix}
        \subseteq
        \img\begin{bmatrix}
            \Id{\xdim} & 0            & 0      & 0 \\
            0          & \abs{\Sigma} & 0      & 0 \\
            0          & \Sigma       & 0      & 0 \\
            0          & \Sigma       & 0      & 0 \\
            0          & 0            & \Omega & 0 \\
            0          & 0            & 0      & 0
        \end{bmatrix}
        = \lim_{\ell \to \infty} \tspace{(\jval{\totalvar}_\ell, \jval{\totalconstr}_\ell)}{\stratum[\sigma, \omega]},
    \end{equation}
    where the last limit is taken over a sequence of constant tangent spaces.
\end{proof}
\begin{theorem}[LIKQ as a tranversality condition]\label{thm:likq-by-transversality}
    For any tuple of functions \(\tuple{\switchfct, \ineqconstr, \eqconstr} \in
    \switchfcts \times \ineqconstrs \times \eqconstrs\)
    one has
    \begin{equation}
        \forall \tuple{x, z} \in \feasible(\switchfct, \ineqconstr, \eqconstr) \colon \text{LIKQ holds at } \tuple{x, z}
        \qquad\Longleftrightarrow\qquad
        \jet{0}(\switchfct, \ineqconstr, \eqconstr) \transversal \stratification.
    \end{equation}
\end{theorem}
Before the proof of Theorem~\ref{thm:likq-by-transversality} is presented here,
the two examples of Section~\ref{sec:abs-smooth} are revisited to provide an
intuition of the relation between transversality \(\jet{0}(\switchfct,
\ineqconstr, \eqconstr) \transversal \stratifiedset\) and the LIKQ condition
of~Definition~\ref{def:likq}.
\begin{example}[LIKQ everywhere]\label{ex:likq-everywhere}

    Figure~\ref{fig:likq-examples}\,\subref{subfig:likq-example} shows the
    jet-space of Example~\ref{ex:likq-everywhere-intro} without the
    \(\jval{z}\)-dimension.
    The feasible points can be recognized in the jet-space as the intersection
    of the stratified set \(\stratifiedset\) with the image of the jet
    \(\jet{0}(\switchfct)\).
    At all points, in particular the feasible ones, the \(\jval{y}\)-axis is
    part of the linearization of the image of the jet \(\jet{0}(\switchfct)\).
    Moreover, the \(\jval{x}\)-axis is part of the tangent space at any
    (feasible) point in a stratum of \(\stratifiedset\).
    Finally, for all feasible points \((\jval{x}, \jval{y}, \jval{c})\) with
    \(\jval{c} \ne 0\) the tangent spaces contain a third linear independent
    direction \((0, 1, \pm 1)\) which completes the combined dimensions to
    \(3\).
    If on the other hand \(\jval{c} = 0\) for a feasible point then it is a
    point on the stratum \(\stratum[(0)]\) and the tangent space is
    only one-dimensional.
    However, then the \(\jval{c}\)-axis is part of the linearization of the jet
    which again ensures transversality.
\end{example}
\begin{figure}
    \newcommand{\samples}{80}
    \newcommand{\opacity}{0.7}
    \newcommand{\domxmax}{1.5}
    \centering \begin{subfigure}[t]{.5\textwidth}
        \centering
        \clearpage{}

\tikzsetnextfilename{\jobname-likq-example}
\begin{tikzpicture}
    \begin{axis}[
xlabel={\(\jval{x}\)},
        ylabel={\(\jval{y}\)},
        zlabel={\(\jval{c}\)},
xmin=-2.5, xmax=2.5,
        ymin=-1.1, ymax=2.3,
        zmin=-2.1, zmax=3.0,
        view={65}{35},
        axis lines=center,
xtick=\empty,
        ytick=\empty,
        ztick=\empty,
        xlabel style={anchor=west},
        ylabel style={anchor=west},
        zlabel style={anchor=south},
]

        \addplot3 [
        domain y=0:2,
        domain=-2:\domxmax,
        surf,
        faceted color=none,
        fill=pink,
        opacity=2*\opacity,
        samples y=\samples,
        samples=\samples,
        unbounded coords=jump,
        x filter/.expression={sin(deg(x)) >= y ? nan : x},
        shader=flat,
        draw opacity=0,
        ] {-y};
        \label{plot:likq-stratification}

        \addplot3 [
        domain y=0:2,
        domain=-2:\domxmax,
        surf,
        faceted color=none,
        fill=pink,
        opacity=2*\opacity,
        samples y=\samples,
        samples=\samples,
        unbounded coords=jump,
        x filter/.expression={-sin(deg(x)) < y ? nan : x},
        shader=flat,
        draw opacity=0,
        ] {y};

        \addplot3 [
        domain=-2:\domxmax,
        draw=red!50!pink,
        samples=\samples,
        very thick,
        x filter/.expression={- sin(deg(x)) <= y ? nan : x},
        ] ({x}, {0}, {0});
        \label{plot:likq-stratum}

        \addplot3 [
        domain y=-1:2,
        domain=-2:\domxmax,
        surf,
        faceted color=none,
        fill=teal,
        opacity=\opacity,
        samples y=\samples,
        samples=\samples,
shader=flat,
        draw opacity=0,
        ] {- sin(deg(x))};
        \label{plot:likq-jet-image}

        \addplot3 [
        domain=-1.1:0,
        samples=4,
        samples y=4,
        ] ({0}, {y}, {0});

        \addplot3 [
        domain y=0:2,
        domain=-2:\domxmax,
        surf,
        faceted color=none,
        fill=pink,
        opacity=2*\opacity,
        samples y=\samples,
        samples=\samples,
        unbounded coords=jump,
        x filter/.expression={sin(deg(x)) < y ? nan : x},
        shader=flat,
        draw opacity=0,
        ] {-y};

        \addplot3 [-stealth,
        domain=\domxmax:2.5,
        samples=4,
        ] ({x}, {0}, {0});

        \addplot3 [
        domain=-2:\domxmax,
        draw=blue!20!teal,
        samples=\samples,
        samples y=0,
        very thick,
] ({x}, {abs(-sin(deg(x)))}, {-sin(deg(x))});
        \label{plot:likq-feasible-set}

        \addplot3 [
        domain y=0:2,
        domain=-2:\domxmax,
        surf,
        faceted color=none,
        fill=pink,
        opacity=2*\opacity,
        samples y=\samples,
        samples=\samples,
        unbounded coords=jump,
        x filter/.expression={- sin(deg(x)) >= y ? nan : x},
        shader=flat,
        draw opacity=0,
        ] {y};

        \addplot3 [
        domain=-2:\domxmax,
        draw=red!50!pink,
        samples=\samples,
        very thick,
        x filter/.expression={-sin(deg(x)) >= y ? nan : x},
        ] ({x}, {0}, {0});

        \addplot3 [draw, black]
        coordinates
        {
            (0., 0., 0.)
            (0., 0., 2.9)
        }
        ;
    \end{axis}
\end{tikzpicture}
\clearpage{}
        \caption{Example with LIKQ at all feasible points.}
        \label{subfig:likq-example}
    \end{subfigure}\hfill
    \begin{subfigure}[t]{.5\textwidth}
        \centering
        \clearpage{}

\tikzsetnextfilename{\jobname-no-likq-example}
\begin{tikzpicture}
    \begin{axis}[
xlabel={\(\jval{x}\)},
        ylabel={\(\jval{y}\)},
        zlabel={\(\jval{c}\)},
xmin=-2.5, xmax=2.5,
        ymin=-1.1, ymax=2.3,
        zmin=-2.1, zmax=3.0,
        view={65}{35},
        axis lines=center,
xtick=\empty,
        ytick=\empty,
        ztick=\empty,
        xlabel style={anchor=west},
        ylabel style={anchor=west},
        zlabel style={anchor=south},
]

        \addplot3 [
        domain y=0:2,
        domain=-2:\domxmax,
        surf,
        faceted color=none,
        fill=pink,
        opacity=2*\opacity,
        samples y=\samples,
        samples=\samples,
        unbounded coords=jump,
        x filter/.expression={-(sin(deg(x)) - x) >= y ? nan : x},
        shader=flat,
        draw opacity=0,
        ] {-y};
        \label{plot:no-likq-stratification}

        \addplot3 [
        domain y=0:2,
        domain=-2:\domxmax,
        surf,
        faceted color=none,
        fill=pink,
        opacity=2*\opacity,
        samples y=\samples,
        samples=\samples,
        unbounded coords=jump,
        x filter/.expression={sin(deg(x)) - x < y ? nan : x},
        shader=flat,
        draw opacity=0,
        ] {y};

        \addplot3 [
        domain=-2:\domxmax,
        draw=red!50!pink,
        samples=\samples,
        very thick,
        x filter/.expression={sin(deg(x)) - x <= y ? nan : x},
        ] ({x}, {0}, {0});
        \label{plot:no-likq-stratum}

        \addplot3 [
        domain y=-1:2,
        domain=-2:\domxmax,
        surf,
        faceted color=none,
        fill=teal,
        opacity=\opacity,
        samples y=\samples,
        samples=\samples,
shader=flat,
        draw opacity=0,
        ] {sin(deg(x)) - x};
        \label{plot:no-likq-jet-image}

        \addplot3 [
        domain=-1.1:0,
        samples=4,
        samples y=4,
        ] ({0}, {y}, {0});

        \addplot3 [
        domain y=0:2,
        domain=-2:\domxmax,
        surf,
        faceted color=none,
        fill=pink,
        opacity=2*\opacity,
        samples y=\samples,
        samples=\samples,
        unbounded coords=jump,
        x filter/.expression={-(sin(deg(x)) - x) < y ? nan : x},
        shader=flat,
        draw opacity=0,
        ] {-y};

        \addplot3 [-stealth,
        domain=\domxmax:2.5,
        samples=4,
        ] ({x}, {0}, {0});

        \addplot3 [
        domain=-2:\domxmax,
        draw=blue!20!teal,
        samples=\samples,
        samples y=0,
        very thick,
] ({x}, {abs(sin(deg(x)) - x)}, {sin(deg(x)) - x});
        \label{plot:no-likq-feasible-set}

        \addplot3 [
        domain y=0:2,
        domain=-2:\domxmax,
        surf,
        faceted color=none,
        fill=pink,
        opacity=2*\opacity,
        samples y=\samples,
        samples=\samples,
        unbounded coords=jump,
        x filter/.expression={sin(deg(x)) - x >= y ? nan : x},
        shader=flat,
        draw opacity=0,
        ] {y};

        \addplot3 [
        domain=-2:\domxmax,
        draw=red!50!pink,
        samples=\samples,
        very thick,
        x filter/.expression={sin(deg(x)) - x >= y ? nan : x},
        ] ({x}, {0}, {0});

        \addplot3 [draw, black]
        coordinates
        {
            (0., 0., 0.)
            (0., 0., 2.9)
        }
        ;
    \end{axis}
\end{tikzpicture}
\clearpage{}
        \caption{Example without LIKQ at \(0\).}
        \label{subfig:no-likq-example}
    \end{subfigure}\caption{Two examples that show different situations with respect to the LIKQ
        condition.
        The pink area (\protect\ref{plot:no-likq-stratification}) indicates the
        stratified set \(\stratifiedset\) and the teal area
        (\protect\ref{plot:no-likq-jet-image}) indicates the image of the jet
        \(\jet{0}(\switchfct)\).
        The intersection of the former two, the feasible set
        \(\feasible(\switchfct)\), is depicted as the blue line
        (\protect\ref{plot:no-likq-feasible-set}).
        Finally, the red line indicates (\protect\ref{plot:no-likq-stratum})
        the stratum \(\stratum[(0)]\).
    }\label{fig:likq-examples}
\end{figure}
\begin{example}[no LIKQ at 0]
Figure~\ref{fig:likq-examples}\,\subref{subfig:no-likq-example} depicts the
    situation of Example~\ref{ex:no-likq-intro}.
    As in Example~\ref{ex:likq-everywhere}, the tangent space of
    \(\stratum[(0)]\) is just the \(\jval{x}\)-axis.
    However, here the linearization of \(\jet{0}(x, y, z)\) at \(0\) spans the
    \(\jval{x}\)-\(\jval{y}\)-plane leaving the \(\jval{c}\)-axis as a linear
    independent dimension.
\end{example}
\begin{proof}
    By definition the transversality \(\jet{0}(\switchfct, \ineqconstr,
    \eqconstr) \transversal \stratification\), holds if and only if for all
    \(\tuple{x, y, z} \in {\jet{0}(\switchfct, \ineqconstr,
    \eqconstr)}^{-1}(\stratifiedset)\)
    \begin{equation}
        \label{eq:transversality-condition-likq}
        \img(\D \jet{0}(\switchfct, \ineqconstr, \eqconstr)(x, y, z)) + \tspace{\jet{0}(\switchfct, \ineqconstr, \eqconstr)(x, y, z)}{\stratum[\sigma, \omega]}
        = \reals^{\totaldim + \totalconstrs},
    \end{equation}
    where \(\stratum[\sigma, \omega]\) is the stratum containing
    \(\tuple{\jval{x}, \jval{y}, \jval{z}, \jval{\switchfct},
    \jval{\ineqconstr}, \jval{\eqconstr}} \coloneqq \jet{0}(\switchfct,
    \ineqconstr, \eqconstr)(x, y, z)\).
    By the definitions of the sets $\stratum[\sigma]$ and $\stratum[\omega]$ in
    Equation~\eqref{eq:def-strata}, one obtains \(\sigma = \sign(\jval{z}) =
    \sign(z)\), \(\omega = \sign(\jval{\ineqconstr})\) and \(y = \abs{z}\).
    Define \(\Sigma \coloneqq \diag(\sigma)\) and \(\Omega \coloneqq
    \diag(\omega)\), then by Lemma~\ref{lem:statum-tspace} the
    condition~\eqref{eq:transversality-condition-likq} is equivalent to
\begin{equation}\label{eq:tansversality-condition-as-matrix-rank}
        \rank
        \begin{bmatrix}
            \Id{\xdim}    & 0               & 0               & \Id{d} & 0            & 0 \\
            0             & \Id{\nswitches} & 0               & 0      & \abs{\Sigma} & 0 \\
            0             & 0               & \Id{\nswitches} & 0      & \Sigma       & 0 \\
            \switchfctdx  & \switchfctdy    & \switchfctdz    & 0      & \Sigma       & 0 \\
            \ineqconstrdx & \ineqconstrdy   & \ineqconstrdz   & 0      & 0            & \Omega \\
            \eqconstrdx   & \eqconstrdy     & \eqconstrdz     & 0      & 0            & 0
        \end{bmatrix}
        = \totaldim + \totalconstrs,
    \end{equation}
    where the trivial last column in the characterization of the tangent spaces
    is left out, since it does not add to the overall rank, and the
    dependencies on the evaluation point \((x, \abs{z}, z)\) is omitted in the
    partial derivatives of \(\switchfct\), \(\ineqconstr\) and \(\eqconstr\)
    for notational ease.
\newcommand{\Tz}{S}\newcommand{\Tg}{(\ineqconstrdy\Sigma + \ineqconstrdz)\Tz}\newcommand{\Th}{(\eqconstrdy\Sigma + \eqconstrdz)\Tz}Multiplication from the left of the above matrix with the full rank matrix
    \begin{equation}
        \begingroup \setlength\arraycolsep{2.6pt}
        \begin{bmatrix}
            \Id{\xdim} & 0                                 & 0                                  & 0   & 0                  & 0                \\
            0          & \Id{\nswitches}                   & 0                                  & 0   & 0                  & 0                \\
            0          & 0                                 & \Id{\nswitches}                    & 0   & 0                  & 0                \\
            0          & -\Tz \switchfctdy                 & -\Tz \switchfctdz                  & \Tz & 0                  & 0                \\
            0          & -\Tg \switchfctdy - \ineqconstrdy & - \Tg \switchfctdz - \ineqconstrdz & \Tg & \Id{\nineqconstrs} & 0                \\
            0          & -\Th \switchfctdy - \eqconstrdy   & - \Th \switchfctdz - \eqconstrdz   & \Th & 0                  & \Id{\neqconstrs} \\
        \end{bmatrix}\!\!,
        \endgroup
    \end{equation}
    where \(\Tz \coloneqq \paren{\Id{\nswitches} - \switchfctdy\Sigma -
    \switchfctdz}^{-1} \in \reals^{\nswitches \times \nswitches}\),
yields
    \begin{equation}
        \rank
        \begin{bmatrix}
            \Id{\xdim} & 0 & 0 & \Id{\xdim} & 0 & 0 \\
            0 & \Id{\nswitches} & 0 & 0 & \abs{\Sigma} & 0 \\
            0 & 0 & \Id{\nswitches} & 0 & \Sigma & 0 \\
            J_{z}(x, z) & 0 & 0 & 0 & \Sigma & 0 \\
            J_{\ineqconstr}(x, z) & 0 & 0 & 0 & 0 & \Omega \\
            J_{\eqconstr}(x, z) & 0 & 0 & 0 & 0 & 0 \\
        \end{bmatrix} = \totaldim + \totalconstrs
    \end{equation}
    as an equivalent characterization of transversality.
    Clearly, the first \(\totaldim = \xdim + \nswitches + \nswitches\) rows of
    this matrix are independent, and thus, the final reformulation of the
    transversality condition in~\eqref{eq:transversality-condition-likq} reads
    \begin{equation}
       \rank \begin{bmatrix}
           J_{z}(x,z) & \Sigma & 0 \\
           J_{\ineqconstr}(x,z) & 0      & \Omega \\
           J_{\eqconstr}(x,z) & 0      & 0
        \end{bmatrix} = \totalconstrs.
    \end{equation}
    This is exactly the LIKQ condition in Definition~\ref{def:likq} if the rows
    that are trivially linear independent from the rest, due to nonzero entries
    in \(\Sigma\) or \(\Omega\), are removed from the matrix.
\end{proof}

\begin{remark}[No addtional constraints]\label{rmk:no-additional-constraints}
    An analogous statement to Theorem~\ref{thm:likq-by-transversality} can be
    derived, when considering the problem class without additional inequality
    and equality constraints.
    In that case the LIKQ condition of Definition~\ref{def:likq} reduces to a
    full rank condition of \(\projactsw{\actswitches(\switchvar)}
    J_\switchvar(\optvar, \switchvar)\).
    Changing of the definition of the feasible and stratified set accordingly,
    i.e., to \(\feasible(\switchfct) \coloneqq \set{(\optvar, \switchvar) \with
    \switchvar = \switchfct(\optvar, \abs{\switchvar}, \switchvar)}\) and
    \begin{equation}
        \stratifiedset
        \coloneqq \Union_{\signature \in \set{-1, 0, 1}^\nswitches}
                  \set{(\jval{\optvar}, \jval{\absvar}, \jval{\switchvar}, \jval{\switchfct)})
                  \with \jval{\switchvar} = \jval{\switchfct}, \jval{\absvar} = \abs{\jval{\switchvar}}, \sign(\jval{\switchvar}) = \signature}
    \end{equation}
    one can use virtually the same arguments as in the proof of
    Theorem~\ref{thm:likq-by-transversality}.
    The only difference being that the last two rows and the last column of the
    matrix in~\eqref{eq:tansversality-condition-as-matrix-rank} are not present.
\end{remark}
\begin{remark}[No free switching variable]\label{rmk:no-free-switching-var}
    Earlier papers on abs-smooth optimization, in particular~\cite{Gri13, GW16,
    WG19}, used a different formalization to represent an abs-smooth function.
    Therein, the functions in the evaluation procedure do not depend explicitly
    on the switching variable \(\switchvar\) but only on its absolute value.
    That is, the representation of \(\absobjfct \in \Cabs(\reals^\xdim)\) reads
    \begin{equation}
        \begin{aligned}
            \absobjfct(\optvar) &= \objfct(\optvar, \abs{\switchvar}),\\
            \switchvar &= \switchfct(\optvar, \abs{\switchvar}).
        \end{aligned}
    \end{equation}
    The same can be done for abs-smooth inequality constraints and abs-smooth
    equality constraints prescribed by functions \(\absobjfct_\ineqconstr\) and
    \(\absobjfct_\eqconstr\) respectively.
    Consequently, in such a formulation the matrices
    in~\eqref{eq:def-jacobians} would not involve the corresponding partial
    derivatives with respect to a third argument.
    However, the definition of the feasible set, the jet and the strata
    in~\eqref{eq:def-feasible},~\eqref{eq:def-jet} and~\eqref{eq:def-strata}
    would need to remain unchanged to properly encode feasibility.
    The then seemingly unjustified input variable \(\switchvar\) of the jet can
    be explained by introducing artificial function \(\psi(\optvar, \absvar,
    \switchvar)\) and considering the reduced structured jet that drops the
    actual value of \(\psi\), c.f.~\cite{Gün08}.
    In the later context on genericity this additional freedom in the possible
    pertubations to the problem description needs to be justified, which will
    be done in Remark~\ref{rmk:genericity-other-settings}.
    In a proof of a theorem analogous to
    Theorem~\ref{thm:likq-by-transversality} basically nothing would change,
    except that all partial derivatives with respect to \(\switchvar\) are
    replaced by \(0\)-matrices of the corresponding dimension.
\end{remark}
 
\section{Genericity of LIKQ}\label{sec:generLIKQ}
As already discussed after Definition~\ref{def:jet-extensions} the
perturbations that the standard jet-trans\-ver\-sality
theorem~\cite[Theorem~7.4.5]{JJT00} for \(\jet{0}(\switchfct, \ineqconstr,
\eqconstr)\) and \(\stratifiedset\) as defined in~\eqref{eq:def-jet}
and~\eqref{eq:def-stratified-set} considers are too general.
In particular, when perturbing \(\switchfct\) in \(\contdiff(\reals^\totaldim;
\reals^\totalconstrs)\) even slightly with respect to the strong Whitney
topology, the result may not be a valid switching function anymore.
In contrast, perturbing the functions \(\switchfct_1, \ldots,
\switchfct_\nswitches\) that define \(\switchfct\)
via~\eqref{eq:def-switch-fct} individually will always result in a valid
switching function.
However, since the signatures of the \(\switchfct_i\), \(i \in
\oneto{\nswitches}\) differ, this requires the use of the structured
jet-transversality Theorem~\ref{thm:structured-jet-transversality}.
\begin{theorem}[LIKQ is a generic assumption]\label{thm:genericity-of-likq}
    Assume \(\smoothness \ge \xdim - \neqconstrs\).
    Then, the set \(\goodconstrs \subset \contdiff(\reals^{\totaldim_1}) \times
    \ldots \times \contdiff(\reals^{\totaldim_\nswitches}) \times \ineqconstrs \times
    \eqconstrs\) defined by
    \begin{equation}
        \goodconstrs \coloneqq \set{\tuple{\switchfct_1, \ldots, \switchfct_\nswitches, \ineqconstr, \eqconstr}
            \with
            \forall \tuple{x, z} \in \feasible(\switchfct, \ineqconstr, \eqconstr)
            \text{ LIKQ holds at } \tuple{x, z}
        }
    \end{equation}
    is an open and dense subset of \(\contdiff(\reals^{\totaldim_1}) \times
    \ldots \times \contdiff(\reals^{\totaldim_\nswitches}) \times
    \contdiff(\reals^\totaldim; \reals^\nineqconstrs) \times
    \contdiff(\reals^\totaldim; \reals^\neqconstrs)\) with respect to the
    strong \(\contdiff\)-Whitney topology.
\end{theorem}

\begin{proof}
    Let \(\nfcts \coloneqq \nswitches + 2\).
    For \(i \in \oneto{\nswitches}\) let \(\totaldim_i \coloneqq \xdim + 2(i-1)\) and
    \(\totaldim_{\nswitches+1} \coloneqq \totaldim_{\nswitches+2} \coloneqq \totaldim\).
    The matrices \(\reals^{\totaldim_i \times \totaldim}\) with
    \begin{equation}
        \Pi_i \begin{bmatrix} x & y & z \end{bmatrix}^T \coloneqq
        \begin{bmatrix} x & y_1 & \ldots & y_{i-1} & z_1 & \ldots & z_{i-1} \end{bmatrix}^T,
        \quad i \in \oneto{\nswitches}
    \end{equation}
    generate the required inputs for \(i\)-th line of the switching function
    \(\switchfct_i\).
    By setting \(\Pi_{\nswitches+1} \coloneqq \Pi_{\nswitches+2} \coloneqq \Id{\totaldim}\)
    and \(\enlarged{\totaldim} \coloneqq \totaldim_1 + \ldots + \totaldim_\nfcts\)
    the combined matrix
    \begin{equation}
        \Pi \coloneqq
        \begin{bmatrix} \Pi_1 \\ \vdots \\ \Pi_{\nfcts} \end{bmatrix}
        \in \reals^{\enlarged{\totaldim} \times \totaldim}
    \end{equation}
    when applied to \((x, y, z) \in \reals^\totaldim\) provides valid inputs
    for the structured evaluation of \(\enlarged{\totalfct} \coloneqq (\switchfct_1, \ldots,
    \switchfct_\nswitches, \ineqconstr, \eqconstr)\) in the sense that
    \begin{equation}
        \enlarged{\totalfct}(\Pi(x, y, z))
        = (\switchfct_1(x),
\ldots,
            \switchfct_\nswitches(x, y_1, \ldots, y_{\nswitches-1}, z_1, \ldots, z_{\nswitches-1}),
            \ineqconstr(x, y, z),
            \eqconstr(x, y, z)
            ).
    \end{equation}
    Moreover, \(\Pi\) is left-invertible, as the last submatrix
    \(\Pi_{\nswitches + 2}\) is an identity on the input, and it holds
    \(\enlarged{\totalfct} \after \Pi = (\switchfct, \ineqconstr, \eqconstr) \in
    \contdiff(\reals^\totaldim; \reals^\totalconstrs)\).
    Since the stratification \(\stratification\) of the \(\stratifiedset\)
    defined in~\eqref{eq:def-stratified-set} is weakly Whitney regular,
    the Assertion~\ref{it:is-stratification} and~\ref{it:is-Whitney-regular} of
    Lemma~\ref{lem:simple-jets} show that \(\tilde{\stratification}\)
    is a Whitney regular stratification of \(\tilde{\stratifiedset}\), where
    \begin{equation}
        \tilde{\stratification}
        \coloneqq \set*{\begin{bmatrix}
                \Pi & 0 \\
                0 & \Id{\totalconstrs}
            \end{bmatrix}
            \stratum[\sigma, \omega] \with \sigma \in \set{-1, 0, 1}^\nswitches, \omega \in \set{0, 1}^\nineqconstrs
        }
        \quad\text{and}\quad
        \tilde{\stratifiedset} \coloneqq
        \begin{bmatrix}
            \Pi & 0 \\
            0 & \Id{\totalconstrs}
        \end{bmatrix} \stratifiedset.
    \end{equation}
    Since the transformation that is applied to each stratum is left-invertible
    by construction, it is bijective onto its image, and hence,
    \(\dim(\largestratification) = \dim(\stratification) \le \xdim + \nswitches
    + \nineqconstrs\).
    Now Theorem~\ref{thm:structured-jet-transversality} shows that
    the set \(\set{(\switchfct_1, \ldots, \switchfct_\nswitches, \ineqconstr,
    \eqconstr) \with \jet{0}(\switchfct_1, \ldots, \switchfct_\nswitches,
    \ineqconstr, \eqconstr) \transversal \tilde{\stratification}}\)
    is an open and dense subset of \(\contdiff(\reals^{\totaldim_1}) \times
    \ldots \times \contdiff(\reals^{\totaldim_\nswitches}) \times
    \contdiff(\reals^\totaldim; \reals^\nineqconstrs) \times
    \contdiff(\reals^\totaldim; \reals^\neqconstrs)\) with respect to the
    strong Whitney topology.
    Further, Lemma~\ref{lem:simple-jets}~\ref{it:is-transversal} ensures that
    \begin{equation}
        \set{(\switchfct_1, \ldots, \switchfct_\nswitches, \ineqconstr, \eqconstr) \with
            \jet{0}(\switchfct_1, \ldots, \switchfct_\nswitches, \ineqconstr, \eqconstr) \transversal \tilde{\stratification}
        }
        =
        \set{(\switchfct_1, \ldots, \switchfct_\nswitches, \ineqconstr, \eqconstr) \with
            \jet{0}((\switchfct, \ineqconstr, \eqconstr)) \transversal \stratification
        },
    \end{equation}
    while Theorem~\ref{thm:likq-by-transversality} provides the required
    connection to the LIKQ condition.
\end{proof}
\begin{remark}[genericity in other settings]\label{rmk:genericity-other-settings}
    In view of Remark~\ref{rmk:no-additional-constraints} and
    Remark~\ref{rmk:no-free-switching-var} one might wonder if the result of
    Theorem~\ref{thm:genericity-of-likq} still holds if the problems are
    formulated without additional inequality and equality constraints, or when
    the functions \(\switchfct\), \(\ineqconstr\) and \(\eqconstr\) are
    formulated without an explicit dependence on the switching variable
    \(\switchvar\).

    In the first case of problem formulations without additional constraints
    \(\ineqconstr\), \(\eqconstr\), the injectivity of the operator \(\Pi\)
    can not be concluded as in the previous proof.
    Similar to the argument in Remark~\ref{rmk:no-free-switching-var}, this can
    be fixed by adding an artificial function \(\psi\) that takes all inputs
    \(\optvar\), \(\absvar\) and \(\switchvar\) and whose value is reduced in a
    structured jet that together with \(\stratifiedset\) encodes feasibility.
    The inputs to \(\psi\) are then again taken care of by an identity block in
    the last \(\totaldim\) rows of \(\Pi\) which ensures the injectivity just
    as before.

    Thus, in both cases the final application of the
    structured-jet-transversality theorem in the proof of
    Theorem~\ref{thm:genericity-of-likq} ensures that the functions that have
    LIKQ at every feasible points form an open and dense set.
    Removing the artificial function \(\psi\) to obtain a clear result can
    simply be done by projecting onto the other components.
    This projection preserves the openness and the density.
\end{remark}
 
\section{Conclusion and Outlook}
\label{sec:conclusion}
For abs-smooth optimization problems, the property LIKQ serves as a
qualification of the nonsmoothness that allows to verify the optimality of a
given point in polynomial time.
The main result of the paper at hand states that requiring LIKQ at all feasible
points of an abs-smooth problem is a generic assumption in that the set of
problems for which this is true is dense and open in the strong Whitney
topology.

Optimality conditions for large classes of nonsmooth optimization problems were
derived in~\cite{GW16} and~\cite{HS20}.
As a next step, future research could aim at developing a topologically
meaningful stationarity definition, i.e., one that corresponds to a topological
change of the sub-level sets, and prove that generically all stationary points
in this sense satisfy an associated non-degenericity condition.
Most likely this can be achieved using similar arguments as the once used in
this paper.
In particular one requires the \(1\)-jet, i.e., a jet the incorporates the
first derivatives of the involved functions and a stratified set that encodes
the topological stationarity condition on images of that jet.
Using again Lemma~\ref{lem:simple-jets} this jet should be relatable to a
structured jet and the application of
Theorem~\ref{thm:structured-jet-transversality} then should provide the
genericity result.
Similar results for other classes of nonsmooth optimization problems have been
published in~\cite{JRS09, DSS12, DJS13, LS22}.
 
\section*{Acknowledgements}
The authors are very grateful Vladimir Shikhman for his thorough introduction
to the theory of jet-transversality and numerous discussions on this subject.

The work on this paper was partly funded by the Deutsche Forschungsgemeinschaft
(DFG, German Research Foundation) under Germany´s Excellence Strategy – The
Berlin Mathematics Research Center MATH+ (EXC-2046/1, project ID: 390685689).

\printbibliography 

\end{document}